\documentclass[12pt]{amsart}
\usepackage{amsmath,amssymb,amsthm}
\usepackage{setspace}
\usepackage{moreverb}
\usepackage{mathrsfs}
\usepackage{url}
\usepackage[all]{xy}
\usepackage{lscape}
\usepackage{tikz}
\usetikzlibrary{arrows}
\usepackage{color}

\newcommand{\ignore}[1]{}

\newtheorem{theorem}{Theorem}[section]
\newtheorem{lemma}[theorem]{Lemma}
\newtheorem{corollary}[theorem]{Corollary}
\newtheorem{proposition}[theorem]{Proposition}
\newtheorem{algorithm}[theorem]{Algorithm}
\theoremstyle{definition}
\newtheorem{definition}[theorem]{Definition}
\newtheorem{example}[theorem]{Example}

\theoremstyle{remark}
\newtheorem{remark}[theorem]{Remark}

\numberwithin{equation}{section}

\input xy
\xyoption{all}

\newcommand{\bP}{{\mathbb{P}}}

\newcommand{\bC}{\mathbb{C}}

\newcommand{\bZ}{{\mathbb{Z}}}

\newcommand{\bQ}{{\mathbb{Q}}}
\newcommand{\bR}{\mathbb{R}}

\definecolor{grey}{rgb}{0.75,0.75,0.75}
\definecolor{orange}{rgb}{1.0,0.5,0.5}
\definecolor{brown}{rgb}{0.5,0.25,0.0}
\definecolor{pink}{rgb}{1.0,0.5,0.5}

\newcommand{\adj}{\mathrm{Adj}}

\newcommand{\fM}{{\mathfrak m}}

\newcommand{\fa}{\mathfrak{a}}

\newcommand{\cC}{{\mathcal C}}

\newcommand{\cO}{{\mathcal O}}

\newcommand{\cR}{{\mathcal R}}

\newcommand{\lra}{{\longrightarrow}}

\newcommand{\lambdab}{\pmb{\lambda}}
\newcommand{\lambdapb}{\pmb{\lambda'}}
\newcommand{\lambdappb}{\pmb{\lambda''}}
\newcommand{\mub}{\pmb{\mu}}
\newcommand{\nub}{\pmb{\nu}}
\newcommand{\mupb}{\pmb{\mu'}}
\newcommand{\muppb}{\pmb{\mu ''}}
\newcommand{\fab}{\pmb{\mathfrak{a}}}
\newcommand{\Reg}{\mathcal{R}}

\newcommand{\A}{\mathfrak{a}}

\newcommand{\Z}{\mathbb{Z}}
\newcommand{\Q}{\mathbb{Q}}
\newcommand{\R}{\mathbb{R}}

\newcommand{\Oc}{\mathcal{O}}

\newcommand{\J}{\mathcal{J}}

\begin{document}

\title[Constancy regions of mixed multiplier ideals]{Constancy regions of mixed multiplier ideals in two-dimensional local rings with rational singularities }

\author[M. Alberich-Carrami\~nana]{Maria Alberich-Carrami\~nana}

\author[J. \`Alvarez Montaner]{Josep \`Alvarez Montaner}

\author[F. Dachs-Cadefau]{Ferran Dachs-Cadefau }

\address{Departament de Matem\`atiques\\
Universitat Polit\`ecnica de Catalunya\\ Av. Diagonal 647, Barcelona
08028, Spain} \email{Maria.Alberich@upc.edu, Josep.Alvarez@upc.edu}

\address{KU Leuven\\ Department of Mathematics \\ Celestijnenlaan 200B box 2400\\
BE-3001 Leuven, Belgium}

\email{Ferran.DachsCadefau@kuleuven.be}

\thanks{All three authors were partially supported by Generalitat de Catalunya 2014 SGR-634 project  and
Spanish Ministerio de Econom\'ia y Competitividad
MTM2015-69135-P. FDC  is also supported by the KU Leuven
grant OT/11/069. MAC is also with the  Institut de Rob\`otica i
Inform\`atica Industrial (CSIC-UPC) and the Barcelona Graduate
School of Mathematics (BGSMath).}



\begin{abstract}
The aim of this paper is to study mixed  multiplier ideals associated to a tuple of
ideals 
in a two-dimensional local ring with a rational singularity.
We are interested in the partition of the real positive orthant
given by the regions where the mixed multiplier ideals are constant.
In particular we reveal which information encoded in a mixed multiplier
ideal determines its corresponding jumping wall and we provide an algorithm to
compute all the constancy regions, and their corresponding mixed multiplier ideals,
in any desired range.

\end{abstract}

\maketitle

\section{Introduction}

Let $X$ be a complex algebraic variety with at most a rational singularity at a point $O\in X$ and
let $\cO_{X,O}$ be the corresponding local ring. The study of {\it multiplier
ideals} $\J(\A^\lambda)$ associated to a given ideal  $\fa
\subseteq \cO_{X,O}$ and a parameter $\lambda\in \bR_{>0}$ has received a lot of attention in recent years
mainly because this is one of the few cases where explicit computations can be performed.
Multiplier ideals form a nested sequence of ideals
$$\Oc_{X,O}\varsupsetneq\J(\A^{\lambda_1})\varsupsetneq\J(\A^{\lambda_2})
\varsupsetneq...\varsupsetneq\J(\A^{\lambda_i})\varsupsetneq...$$
and the rational numbers $0< \lambda_1 < \lambda_2 < \cdots$ where
the multiplier ideals change are called \emph{jumping numbers}.
An explicit formula for the set of {jumping numbers} of
a simple complete ideal or an irreducible plane curve has been given by J\"arviletho \cite{Jar11} and Naie \cite{Nai09}
in the case that $X$ is smooth.  More generally, Tucker  gives in \cite{Tuc10} an algorithm
to compute the jumping numbers of any complete ideal when $X$ has a rational singularity.
The approach given by the authors of this manuscript in \cite{ACAMDC13} is
an algorithm that computes sequentially the chain of multiplier ideals. More precisely,
given any jumping number we may give an explicit description of the corresponding multiplier ideal that,
in turn, it allows us to compute the next jumping number.

\vskip 2mm

Given a tuple of ideals
$\fab:=\left(\fa_1,\dots,\fa_r \right)\subseteq
(\cO_{X,O})^r$ and a point
${\lambdab}:=(\lambda_1,\dots,\lambda_r) \in \R_{\geqslant0}^r$,
we may consider the associated {\it mixed multiplier ideal}
$\J\left(\fab^{\lambdab}\right):=\J\left(\fa_1^{\lambda_1}\cdots\fa_r^{\lambda_r}\right)$
that is nothing but a natural extension of the notion of multiplier ideal to this context.
The main differences that we encounter in this setting is that, whereas the multiplier ideals come
with the set of associated jumping numbers, the mixed multiplier ideals come with
a set of {\it jumping walls}. Roughly speaking, the positive orthant $\R_{\geqslant0}^r$ can be decomposed
in a finite set of {\it constancy regions} where any two points in that constancy region
have the same mixed multiplier ideal. These regions are described by rational polytopes whose boundaries are the
aforementioned jumping walls and consist of  points where the mixed multiplier ideal changes.

\vskip 2mm

The case of mixed multiplier ideals has not received as much attention as multiplier ideals.
Libgober and Musta\c{t}\v{a} \cite{LM11} studied  properties of the jumping wall
associated to the constancy region of the origin $\lambdab_0=(0,...,0)$.
Naie in \cite{Nai13} uses mixed multiplier ideals in order to study the irregularity of abelian coverings
of smooth projective surfaces. En passant, he describes a nice property that jumping walls must satisfy.
Cassou-Nogu\`es and Libgober study in \cite{CNL11,CNL14} analogous notions to  mixed multiplier ideals and jumping
walls, the {\it ideals of quasi-adjunction} and {\it faces of quasi-adjunction} (see \cite{Lib02}), 
associated to germs of plane curves. In \cite[Proposition 2.2]{CNL11}, they describe some methods for
the computation of the regions. Moreover, they provide relations between faces of quasi-adjunction and other
invariants such as the Hodge spectrum or the Bernstein-Sato ideals.
Their methods are refined in \cite{CNL14}, where they provide an explicit description of the jumping walls.

\vskip 2mm

Closely related to multiplier ideals we have the so-called {\it test ideals} in positive characteristic.
 P\'erez \cite{Per13} studied the constancy regions of mixed test ideals and described the corresponding 
 jumping walls  using $p$-fractals.

\vskip 2mm

The aim of this manuscript is to extend the methodology that we developed in \cite{ACAMDC13}
in order to provide an algorithm that allow us to compute all the {\it constancy regions}
in the positive orthant for any tuple of ideals. The paper is structured as follows:
In Section 2 we introduce the theory of mixed multiplier ideals in a very detailed way.
In Section 3 we develop the technical results that lead to the main result of the paper. Namely, in
Theorem \ref{region} we provide a formula to compute the region associated to any point in the positive
orthant. This formula leads to a very simple algorithm (see Algorithm \ref{A2}) that
computes all the constancy regions.
Finally, in Section 4 we extend the notion of {\it minimal jumping divisor} introduced
in \cite{ACAMDC13} to the context of mixed multiplier ideals. In particular, the description
of these divisors is really useful in the proof of the key technical results of Section 3.

\vskip 2mm

The results of this work are part of the Ph.D. thesis of the third author \cite{DC16}.
One may find there some extra details of properties of mixed multiplier ideals as well
as many examples that illustrate our methodology.

\vskip 2mm

{\it Acknowledgements:} This project began during a research stay of the third author at the Institut de 
Math\'ematiques de Bordeaux. He would like to thank
Pierrette Cassou-Nogu\`es  for her support and hospitality.

\newpage


\section{Mixed multiplier ideals}
Let $X$ be a normal surface and $O$ a point where $X$ has at worst a
rational singularity. Namely, for any desingularization $\pi: X'
\rightarrow X$  the stalk at $O$  of the higher direct image $R^1
\pi_* \cO_{X'}$ is zero. For more insight  on the  theory of
rational singularities we refer to the seminal papers by  Artin
\cite{Art66} and  Lipman \cite{Lip69}.

\vskip 2mm

Consider a common {\em log-resolution} of a set of non-zero ideal
sheaves $\fa_1,\dots,\fa_r \subseteq \cO_{X,O}$. Namely, a birational
morphism $\pi: X' \rightarrow X$ such that

\begin{enumerate}
\item[$\cdot$] $X'$ is smooth,
\item[$\cdot$]  For $i=1,\dots , r$ we have $\fa_i\cdot\cO_{X'} = \cO_{X'}\left(-F_i\right)$ for some effective
Cartier divisor $F_i$,
\item[$\cdot$] $\sum_{i=1}^r F_i+E$ is a divisor with simple normal crossings where $E = Exc\left(\pi\right)$ is the exceptional locus.
\end{enumerate}

Since the point $O$ has (at worst) a rational singularity, the
exceptional locus $E$ is a tree of smooth rational curves
$E_1,\dots,E_s$. The divisors $F_i=\sum_j e_{i,j} E_j$ are integral
divisors in $X'$ which can be decomposed  into their  {\it
exceptional} and {\it affine} part according to the support, i.e.
$F_i=F_i^{\rm exc} + F_i^{\rm aff}$ where
$$ F_i^{\rm exc}= \sum_{j=1}^s e_{i,j} E_j \hskip 5mm {\rm and} \hskip 5mm  F_i^{\rm aff}= \sum_{j=s+1}^t e_{i,j} E_j.$$
Whenever $\fa_i$ is an $\fM$-primary ideal\footnote{Here $\fM =
\fM_{X,O} \subseteq \cO_{X,O}$ is the maximal ideal of the local
ring $\cO_{X,O}$ at $O$.  An $\fM$-primary ideal is
 identified with an ideal sheaf that equals $\cO_X$ outside the
point $O$.}, the divisor $F_i$ is just supported on the exceptional
locus. i.e. $F_i=F_i^{\rm exc}$.

\vskip 2mm

For any exceptional component $E_j$, we define the {\em excess} (of
$\fa_i$) at $E_j$ as $\rho_{i,j} = - F_i \cdot E_j$. We also recall
the following notions that will play a special role:
\begin{enumerate}
\item[$\cdot$] A component $E_j$ of $E$ is a {\em rupture} component if
it intersects at least three more components of $E$ (different from
$E_j$).

\item[$\cdot$] We say that $E_j$ is {\em dicritical} if $\rho_{i,j} > 0$ for some $i$. By
\cite{Lip69},
they correspond to {\it Rees valuations}. Non-exceptional components
also correspond to {\it Rees valuations}.
\end{enumerate}

\vskip 2mm

\subsection{Complete ideals and antinef divisors}\label{unloading}

Throughout this work we will heavily use the one to one
correspondence, given by Lipman in \cite[\S 18]{Lip69}, between {\it
antinef divisors} in  $X'$ and {\it complete ideals} in $\Oc_{X,O}$.
First recall that given an effective $\bQ$-divisor $D=\sum d_i E_i $
in $X'$  we may consider its associated (sheaf) ideal
$\pi_{\ast}\Oc_{X'}(-D):= \pi_{\ast}\Oc_{X'}(-\lceil D\rceil)$. Its
stalk at $O$ is
$$I_D:=\{ f\in \Oc_{X,O} \hskip 2mm | \hskip 2mm v_i(f)\geqslant \lceil d_i \rceil \hskip 2mm \text{for all} \hskip 2mm  E_i \leqslant D \}$$

\noindent This is a complete ideal of $\Oc_{X,O}$ which is
$\fM$-primary whenever $D$ has exceptional support.

\vskip 2mm

An {\it antinef} divisor is an effective divisor $D$ in $X'$  with
integral coefficients such that $-D\cdot E_i \geqslant 0$,  for
every exceptional prime divisor $E_i$, $i=1,\dots,s$. The affine
part of  $D$ satisfies $ D^{\rm aff}\cdot E_i \geqslant 0$ therefore
$D$ is antinef whenever $- D^{\rm exc}\cdot E_i \geqslant D^{\rm
aff}\cdot E_i$. One of the advantages to work with antinef divisors
is that they provide a simple characterization for the inclusion (or
strict inclusion) of two given complete ideals. Namely, given two
antinef divisors $D_1,D_2$ in $X'$ we have $\pi_*
\Oc_{X'}(-D_1)\supseteq \pi_* \Oc_{X'}(-D_2)$  if and only if
${D_1}\leqslant D_2$. The strict inclusion is satisfied if and only
if ${D_1} < D_2$. For non-antinef divisors we can only claim the
inclusion $\pi_* \Oc_{X'}(-D_1) \supseteq \pi_* \Oc_{X'}(-D_2)$
whenever $D_1\leqslant D_2$.

\vskip 2mm

In general we may have different $\bQ$-divisors defining the same
ideal. In this case we will say that they are {\it equivalent}. To
find a representative in the equivalence class of a given divisor
$D$ we will consider its so-called {\it antinef closure}. This is
the unique minimal integral antinef divisor $\widetilde{D}$
satisfying $\widetilde{D}\geqslant D$. To compute the antinef
closure we  use an inductive procedure called {\it unloading} that
has been described by Enriques \cite[IV.II.17]{EC15}, Laufer
\cite{Lau72}, Casas-Alvero \cite[\S 4.6]{Cas00} or Reguera
\cite{Reg96} among others. For completeness we briefly recall the
version described in \cite{ACAMDC13}:

\vskip 2mm

{\bf Unloading procedure:}   Let $D$ be any $\bQ$-divisor in $X'$.
Its {\it excess} at the exceptional prime divisor $E_i$ is the
integer $\rho_i= -\lceil D \rceil \cdot E_i$. Denote by $\Theta$ the
set of exceptional components $E_i \leqslant D$ with negative
excesses, i.e. $$\Theta:= \{E_i \leqslant D_{\rm exc} \hskip 2mm |
\hskip 2mm \rho_i= -\lceil D \rceil \cdot E_i <0 \}.$$ To {\it
unload values} on this set is to consider the new divisor
$$D'= \lceil D \rceil + \sum_{E_i \in \Theta} n_i E_i,$$ where $n_i=
\left \lceil \frac {\rho_i}{E_i^2} \right \rceil$. In other words,
$n_i$ is the least integer number such that
$$(\lceil D\rceil+n_i E_i)\cdot E_i= -\rho_i + n_i E_i^2 \leqslant 0.$$

\subsection{Mixed multiplier ideals}\label{sec:mmi}

Given a tuple of ideals $\fab:=\left(\fa_1,\dots,\fa_r \right)\subseteq
(\cO_{X,O})^r$ and a point
$\mathbf{\lambdab}:=(\lambda_1,\dots,\lambda_r) \in \R_{\geqslant0}^r$,
the corresponding {\it mixed multiplier ideal} is defined
as\footnote{By an abuse of notation, we will also denote
$\J\left(\fab^{\bf \lambda}\right)$ its stalk at $O$ so we will omit
the word ''sheaf'' if no confusion arises.}
\begin{equation*} 
\J\left(\fab^{\mathbf{\lambda}}\right):=\J\left(\fa_1^{\lambda_1}\cdots\fa_r^{\lambda_r}\right)
= \pi_*\cO_{X'}\left(\left\lceil K_{\pi} - \lambda_1 F_1- \cdots -
\lambda_r F_r \right\rceil\right)
\end{equation*}

\noindent where the {\it relative canonical divisor} $$ K_{\pi}=
\sum_{i=1}^s k_j E_j $$ is a $\bQ$-divisor in $X'$ supported on the
exceptional locus $E$ and, due to the fact that the matrix of
intersections $\left(E_i\cdot E_j\right)_{1\leqslant i,j \leqslant
s}$ is negative-definite,  it is characterized by the property
\begin{equation} \label{eq-relative-canonical}
\left(K_{\pi}+E_i\right)\cdot E_i  = -2
\end{equation}
for every exceptional component $E_i$, $i=1,\dots,s$. As usual
$\lfloor \cdot \rfloor$ and $\lceil \cdot \rceil$ denote the
operations that take the {\em round-down} and {\em round-up}  of a
given $\bQ$-divisor.

\vskip 2mm

Whenever we only consider a single ideal $\fa:=(\fa_1) \subseteq
\cO_{X,O}$ we recover the usual notion of {\it multiplier ideal} and
is not difficult to check out that mixed multiplier ideals satisfy
analogous properties. For example, the definition of mixed
multiplier ideals is independent of the choice of log resolution,
they are complete ideals and are invariants up to integral closure
so we can always assume that the ideals $\fa_1,\dots,\fa_r$ are
complete. For a detailed overview  we refer to the book of
Lazarsfeld \cite{Laz04}.

\vskip 2mm

\begin{remark} \label{line}
The mixed multiplier ideals of a tuple $\fab=\left(\fa_1,\dots,\fa_r
\right)\subseteq (\cO_{X,O})^r$ contained in the ray passing through
the origin $O$ in the direction of a vector $(u_1,\dots,u_r)\in
\Q_{\geqslant 0}^r$ are the usual multiplier ideals of the ideal
$\fa_1^{\alpha u_1}\cdots\fa_r^{\alpha u_r}$ with a convenient
$\alpha \in \Z_{>0}^r$ such that $\alpha\cdot u_i\in\Z$ for all $i$.

\end{remark}

\vskip 2mm

From the definition of mixed multiplier ideals one can easily deduce
properties on the contention of the ideal corresponding to a fixed
point $\lambdab \in \R_{\geqslant 0}^r$ with respect to those ideals
of points in its neighborhood.  In the sequel, $B_\varepsilon(\lambdab)$ will denote the Euclidean
open ball centered in $\lambdab$ with radius $\varepsilon >0$.  The following properties should be
well-known to experts but we collect them here for completeness. For a detailed proof we refer to \cite{DC16}.

\vskip 2mm

$\bullet$ {\bf Positive orthant properties: }

\begin{enumerate}
\item[$\cdot$] We have $\J(\fab^{\lambdab})\supseteq\J(\fab^{\lambdapb})$ for any
$\lambdapb\in\lambdab+\R^r_{\geqslant 0}$.

\vskip 2mm

\item[$\cdot$] We have $\J(\fab^{\lambdab})=\J(\fab^{\lambdapb})$ for any
$\lambdab'\in\left(\lambdab+\R^r_{\geqslant 0}\right)\cap
B_\varepsilon(\lambdab)$ with $\varepsilon>0$ small enough.

\vskip 2mm

\item[$\cdot$] Let $\lambdapb\in\lambdab+\R^r_{\geqslant 0}$ be a point such that $\J(\fab^{\lambdab})=\J(\fab^{\lambdapb})$.
Then, $\J(\fab^{\lambdab})=\J(\fab^{\lambdappb})$  for any
${\lambdappb}\in\left(\lambdab+\R^r_{\geqslant
0}\right)\cap\left(\lambdapb-\R^r_{\geqslant 0}\right)$.
\end{enumerate}

$\bullet$ {\bf Negative orthant properties: }

\begin{enumerate}

\item[$\cdot$] Let $\lambdapb\in \lambdab-\R^r_{\geqslant 0}$ be a point such that
$\J\left(\fab^{\mupb}\right)\varsupsetneq\J\left(\fab^{\lambdab}\right),$
for any $\mupb\neq \lambdab$ in the segment $\overline{\lambdab \lambdapb}$.
Then, any $\lambdappb\in \lambdab-\R^r_{\geqslant 0}$ also satisfy
$\J\left(\fab^{\muppb}\right)\varsupsetneq\J\left(\fab^{\lambdab}\right),$
for any $\muppb \neq \lambdab$ in the segment $\overline{\lambdab \lambdappb}$.

\vskip 2mm

\item[$\cdot$] We have $\J\left(\fab^{\pmb{\lambdapb}}\right)=\J\left(\fab^{\pmb{\lambdappb}}\right)$
for any $\pmb{\lambdapb} , \pmb{\lambdappb} \in \left(\lambdab -
\R_{\geqslant0}^r\right) \cap B_{\varepsilon}(\lambdab)$ with
$\varepsilon>0$ small enough.
\end{enumerate}

\vskip 2mm

The above results give us some understanding of the behavior of the
mixed multiplier ideals in the positive and negative orthants of a
given point $\lambdab\in\R^r_{\geqslant 0}$ . Indeed, we can give
the following result for the rest of points in a small neighborhood
of $\lambdab$.

\vskip 2mm

$\bullet$ {\bf Points in a small neighborhood: }
The mixed multiplier ideal associated to some $\lambdab \in
\R_{\geqslant0}^r$ is the smallest among the mixed multiplier ideals
in a small neighborhood. That is, we have
$\J\left(\fab^{\lambdapb}\right)\supseteq
\J\left(\fab^{\lambdab}\right),$ for any $\lambdapb \in
B_{\varepsilon}(\lambdab)$ and $\varepsilon >0$ small enough.

\subsection{Jumping Walls}

The most significative difference that we face when dealing with
mixed multiplier ideals is that, whereas the usual multiplier ideals
come with an attached set of numerical invariants, the {\em jumping
numbers} (see \cite{ELSV04}), the corresponding notion for mixed
multiplier ideals is more involved and is described in terms of the
so-called {\em jumping walls} that we will introduce next. As these
notions are based on the contention of multiplier ideals it is then
natural to consider the following:

\vskip 2mm

\begin{definition}\label{region}
Let $\fab:=\left(\fa_1,\dots,\fa_r \right)\subseteq (\cO_{X,O})^r$ be a tuple
of ideals. Then, for each $\lambdab\in \R_{\geqslant0}^r$, we
define:

\vskip 2mm

$\cdot$ The {\it region} of $\lambdab $: \hskip 21mm
$\mathcal{R}_{\fab}\left(\lambdab\right)=\left\{\lambdapb \in
\R_{\geqslant 0}^r \hskip 2mm \left|
 \hskip 2mm \J\left(\fab^{\lambdapb}\right)\supseteq\J\left(\fab^{\lambdab}\right)\right\}\,\right. $.

 \vskip 2mm

$\cdot$ The {\it constancy region} of $\lambdab$: \hskip 3mm
$\mathcal{C}_{\fab}\left(\lambdab\right)=\left\{\lambdapb  \in
\R_{\geqslant 0}^r \hskip 2mm \left| \hskip 2mm
\J\left(\fab^{\lambdapb}\right)=\J\left(\fab^{\lambdab}\right)\right\}\,\right. $.

\end{definition}

\begin{remark}

For a single ideal $\fa \subseteq \cO_{X,O}$, the usual multiplier
ideals form a discrete nested sequence of ideals
$$\Oc_{X,O}\supseteq\J(\A^{\lambda_0})\varsupsetneq\J(\A^{\lambda_1})\varsupsetneq \J(\A^{\lambda_2})\varsupsetneq...\varsupsetneq\J(\A^{\lambda_i})\varsupsetneq...$$
indexed by an increasing sequence of rational numbers $0=\lambda_0 <
\lambda_1 < \lambda_2 < \ldots$, the aforementioned  {\it jumping
numbers}, such that for any $\lambda \in [\lambda_i,\lambda_{i+1})$
it holds
$$\J(\A^{\lambda_i})=\J(\A^\lambda)\varsupsetneq\J(\A^{\lambda_{i+1}}).$$
Therefore, the { region} and the constancy region of $\lambda $ are
respectively  $\mathcal{R}_{\fa}(\lambda)= [\lambda_0,
\lambda_{i+1})$ and $\mathcal{C}_{\fa}(\lambda)= [\lambda_i,
\lambda_{i+1})$.

\end{remark}

From now on we will consider $\bR^r_{\geqslant 0}$ and its subsets endowed
with the subspace topology from the Euclidean topology of $\bR^r$. Thus, any region
$\mathcal{R}_{\fab}\left(\lambdab\right)$ is an open neighborhood of $\lambdab \in \bR^r_{\geqslant 0}$
by properties of multiplier ideals in the neighborhood of a given point. Clearly, we have
$\mathcal{R}_{\fab}\left(\lambdab\right) \supseteq \mathcal{C}_{\fab}\left(\lambdab\right) \ni \lambdab$
and the constancy regions are topological varieties of dimension $r$ with boundary.

\vskip 2mm

The property that relates two points $\lambdab, \lambdapb \in \bR^r_{\geqslant 0}$ whenever
$\J\left(\fab^{\lambdapb}\right)=\J\left(\fab^{\lambdab}\right)$ defines an equivalence relation
in $\bR^r_{\geqslant 0}$, whose classes are the constancy regions. Hence the constancy regions
provide a partition of the positive orthant and any bounded set intersects only a finite number
of them, due to the definition of mixed multiplier ideals in terms of a log-resolution.

\vskip 2mm

There is a partial ordering on the constancy regions: $\mathcal{C}_{\fab}\left(\lambdapb\right) \preccurlyeq
\mathcal{C}_{\fab}\left(\lambdab\right)$ if and only if
$\J\left(\fab^{\lambdapb}\right) \supseteq \J\left(\fab^{\lambdab}\right)$. Equivalently, if and only if
$\lambdapb \in \mathcal{R}_{\fab}\left(\lambdab\right)$ (which is also equivalent to
$\mathcal{C}_{\fab}\left(\lambdapb\right) \subseteq \mathcal{R}_{\fab}\left(\lambdab\right)$
or to $\mathcal{R}_{\fab}\left(\lambdapb\right) \subseteq \mathcal{R}_{\fab}\left(\lambdab\right)$).
Notice that the minimal element is the constancy region $\mathcal{C}_{\fab}\left(\lambdab_0\right)$ of the origin
$\lambdab_0=(0,\dots,0)$. One of the aims of this work is to provide a set of points which includes at least one representative
for each constancy region\footnote{For multiplier ideals we have a total order on the constancy regions, and the
representative that we take is simply the corresponding jumping number.}. These points will be taken over
the boundary of regions
$\mathcal{R}_{\fab}\left(\lambdab\right)$ associated to some  $\lambdab $,
i.e. the points where we have a change in the corresponding mixed
multiplier ideals. Taking into account the behavior of these ideals
in the neighborhood of a given point, we introduce the notion of
jumping point.

\begin{definition}\label{def:JP}
Let $\fab:=\left(\fa_1,\dots,\fa_r \right)\subseteq (\cO_{X,O})^r$ be a tuple
of ideals. We say that $\lambdab \in \R^r_{\geqslant 0}$ is a
\emph{jumping point} of $\fab$ if
$\J\left(\fab^{\lambdapb}\right)\varsupsetneq\J\left(\fab^{\lambdab}\right)$
for all $\lambdapb \in \{\lambdab - \R_{\geqslant0}^r\} \cap
B_{\varepsilon}(\lambdab)$ and $\varepsilon >0$ small enough.
\end{definition}

It follows from the definition of mixed multiplier ideals that the
jumping points $\lambdab \in \bR^r_{\geqslant 0}$   must lie on
hyperplanes of the form
\begin{equation} \label{hyperplanes}
 H_i: \hskip 1mm e_{1,i} z_1+ \cdots +  e_{r,i} z_r= \ell_i + k_j
 \hskip 8mm i=1,\dots,s
\end{equation} where $\ell_i \in \bZ_{>0}$. In particular,
each hyperplane $H_i$ is associated to an exceptional divisor $E_i$.
Therefore, the region $\cR_{\fab}(\lambdab)$ associated to a point
$\lambdab\in \bR_{\geqslant 0}^r$ is a {\it rational convex
polytope} defined  by
$$  e_{1,i} z_1+ \cdots +  e_{r,i} z_r < \ell_i + k_i,$$
 i.e. the minimal region in
$\bR^r_{\geqslant 0}$ described by these inequalities, for suitable $\ell_i$..

\begin{definition}
Let $\fab:=\left(\fa_1,\dots,\fa_r \right)\subseteq (\cO_{X,O})^r$ be a tuple
of ideals. The {\it jumping wall} associated to $\lambdab \in
\bR^r_{\geqslant 0}$ is the boundary of the region
$\cR_{\fab}(\lambdab)$. One usually refers to the jumping wall of
the origin as the {\it log-canonical wall}.
\end{definition}

Notice that the facets of the jumping wall of $\lambdab \in
\bR^r_{\geqslant 0}$ are also rational convex polytopes supported on
 the hyperplanes $H_i$ considered in equation
(\ref{hyperplanes}) that provide the minimal region. We will refer
to them as the {\it supporting hyperplanes} of the jumping wall.

\vskip 2mm

\begin{remark}
Whenever we intersect the jumping walls of a tuple ${\fab=\left(\fa_1,\dots,\fa_r \right)\subseteq (\Oc_{X,O})^r}$
with a
ray from the origin in the
direction of a vector $(u_1,\dots,u_r)\in \Q_{\geqslant 0}^r$, we
obtain (conveniently scaled) the jumping numbers of the ideal $\fa_1^{\alpha u_1}\cdots\fa_r^{\alpha u_r}$ 
with $\alpha\cdot u_i\in\Z$ for all $i$. In particular, the intersections of
the coordinate axes with the jumping walls provide the jumping
numbers of the ideals $\fa_i$,  $i=1,\dots,r$.

\end{remark}

\vskip 2mm

Now we turn our attention to the constancy region of a given point
$\lambdab \in \bR_{\geqslant 0}^r$. In general the constancy region
$\cC_{\fab}(\lambdab)$
is not necessarily a convex polytope. Its boundary is entirely formed by
jumping points and it has two components. Roughly speaking, the {\it inner}
part of the boundary is $\cC_{\fab}(\lambdab) \backslash \cC_{\fab}(\lambdab)^\circ $,
i.e. the non-interior points of $\cC_{\fab}(\lambdab)$, which are the points in $\cC_{\fab}(\lambdab)$ closest
to the origin $\lambdab_0$.
The {\it outer} part is $\overline{\cC_{\fab}(\lambdab)} \backslash \cC_{\fab}(\lambdab)$
formed from the points in the adherence of $\cC_{\fab}(\lambdab)$
which are not in the constancy region, which are the points in $\cC_{\fab}(\lambdab)$ further away
from the origin $\lambdab_0$. Notice that this later component is contained
in the boundary of the region $\cR_{\fab}(\lambdab)$. In particular
the facets of the outer boundary
of the constancy region $\cC_{\fab}(\lambdab)$ are contained in the
facets of the corresponding region so they have the same supporting
hyperplanes. However, it will be important to distinguish the outer facets
of $\cC_{\fab}(\lambdab)$ from the facets of $\cR_{\fab}(\lambdab)$
and it is for this reason that we will refer to them as {\it
$\cC$-facets}. Namely, a $\cC$-facet of $\cC_{\fab}(\lambdab)$ is the intersection of the boundary of a connected component of
$\cC_{\fab}(\lambdab)$ with a supporting hyperplane of $\cR_{\fab}(\lambdab)$. Indeed, every facet of a jumping wall decomposes into
several $\cC$-facets associated to different mixed multiplier
ideals.

\begin{figure}[ht!!]
  \begin{center}
  \begin{minipage}{0.4\textwidth}
  \medskip
  \includegraphics[width=.7\textwidth]{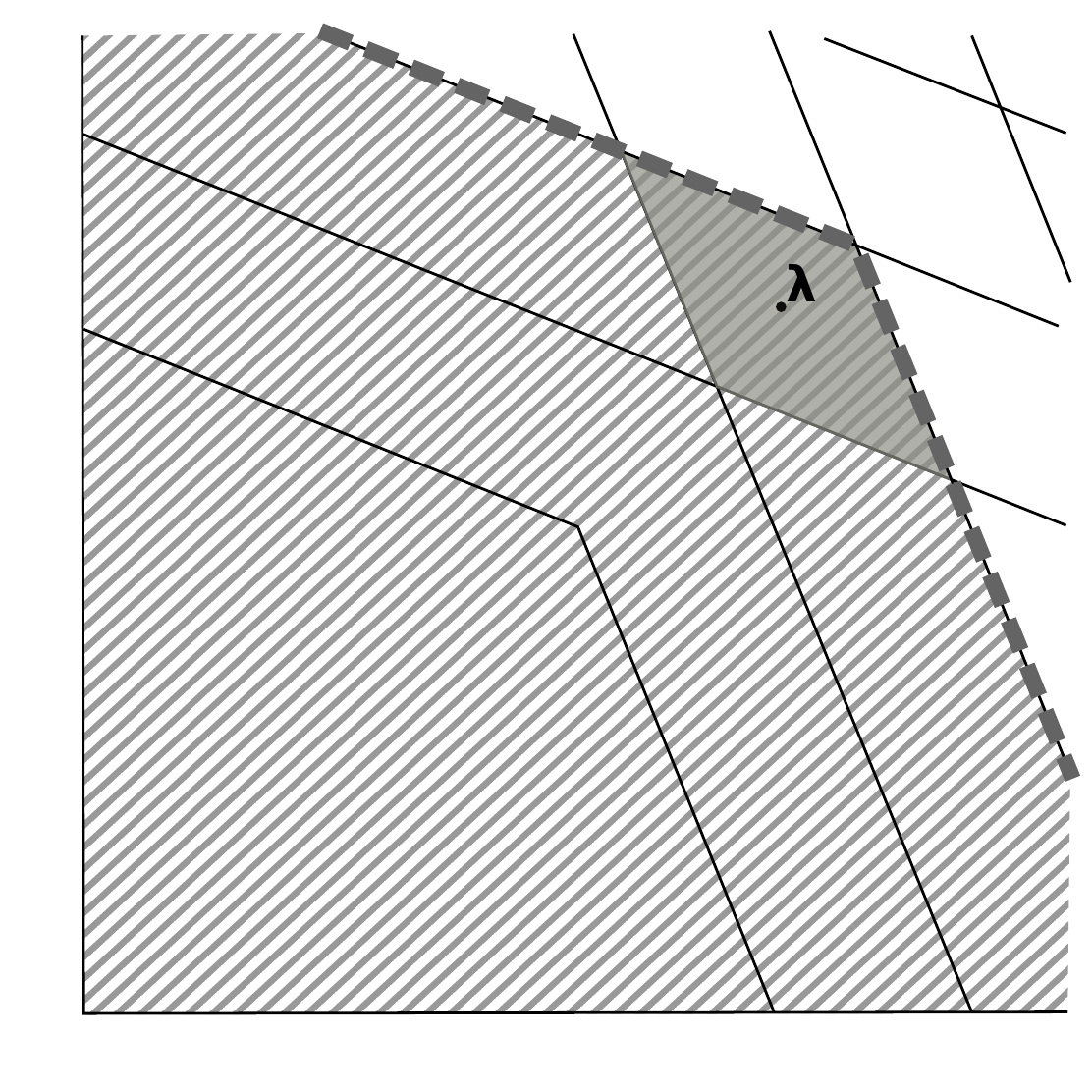}
  \label{fig:JW}
  \end{minipage}
  \begin{minipage}{0.4\textwidth}
  \medskip
  \includegraphics[width=.7\textwidth]{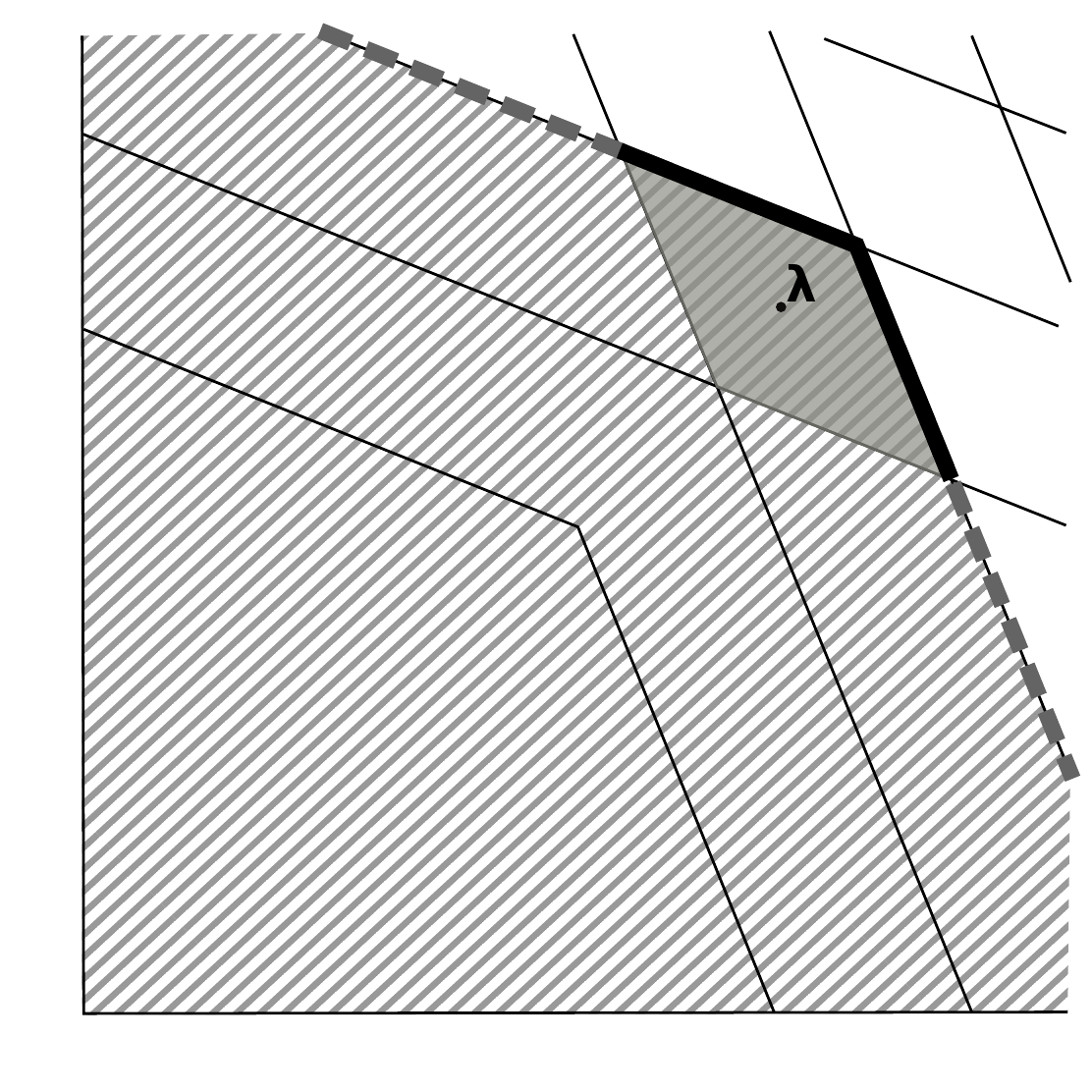}
  \label{fig:Cfacet}
  \end{minipage}
  \end{center}

\caption{On the left, an example of
$\mathcal{R}_{\fab}\left(\lambdab\right)$ in striped gray lines and
the jumping wall associated to $\lambdab$ in thick dotted gray
lines. On the right, the corresponding
$\mathcal{C}_{\fab}\left(\lambdab\right)$ in gray and the
corresponding $\cC$-facets in thick black lines.}

\end{figure}

\vskip 2mm

\begin{remark}
It follows from its definition that the region
$\cR_{\fab}(\lambdab)$  associated to any given point is connected. We do not know whether the
same property is satisfied by the constancy region $\cC_{\fab}(\lambdab)$.
\end{remark}

\section{An algorithm to compute jumping numbers and multiplier
ideals}\label{sec:algorithm}

In \cite{ACAMDC13} we developed  a very simple algorithm to
construct sequentially the chain of multiplier ideals
$$\Oc_{X,O}\supseteq\J(\A^{\lambda_0})\varsupsetneq\J(\A^{\lambda_1})\varsupsetneq \J(\A^{\lambda_2})\varsupsetneq...\varsupsetneq\J(\A^{\lambda_i})\varsupsetneq...$$
associated to a single ideal $\fa \subseteq \cO_{X,O}$. The key
point is the fact proved in \cite[Theorem 3.5]{ACAMDC13} that, given
any $ {\lambda}'\in \bR_{\geqslant 0}$, the consecutive jumping
number is
$$\lambda= \min_i\left\{\frac{k_i+1+
e_i^{\lambda'}}{e_i}\right\},$$ where $D_{{\lambda'}}=\sum
e_i^{{\lambda'}} E_i$ is the antinef closure of $\left\lfloor
\lambda' F  - K_{\pi} \right\rfloor$. In particular, the algorithm
relies heavily on the unloading procedure described in Section
\ref{unloading}.

\vskip 2mm

The goal in this work is to adapt and extend the aforementioned
methods to compute the constancy regions of a tuple of ideals
$\fab:=\left(\fa_1,\dots,\fa_r \right)\subseteq \left(\cO_{X,O}\right)^r$ and describe the
corresponding mixed multiplier ideals. We start by fixing a common
log-resolution $\pi: X' \lra X$ of $\fab$. Then we have to consider
the relative canonical divisor $K_\pi = \sum_{i=1}^s k_j E_j $ and
the divisors $F_i$ in $X'$ such that $\fa_i\cdot \cO_{X'}=
\cO_{X'}(-F_i)$ decomposed as
$$ F_i=F_i^{\rm exc}+F_i^{\rm aff}= \sum_{j=1}^s e_{i,j} E_j + \sum_{j=s+1}^t e_{i,j} E_j $$
in terms of its exceptional and affine support.

\vskip 2mm

\newpage

As in the case treated in \cite{ACAMDC13}, the key point of our
method is how to compare the complete ideals defined by an antinef
and a non-antinef divisor. First we recall the following result.

\vskip 2mm

\begin{proposition}\cite[Corollary
3.4]{ACAMDC13}\label{semicont} Let $D_1,D_2$ be two divisors in $X'$
such that $D_1\leqslant D_2$. Then:

\begin{itemize}

\item[i)] $\pi_* \Oc_{X'}(-D_1)= \pi_* \Oc_{X'}(-D_2)$ if and only
if $\widetilde{D_1}\geqslant D_2$.

\item[ii)] $\pi_* \Oc_{X'}(-D_1) \varsupsetneq \pi_* \Oc_{X'}(-D_2)$ if and only if
$v_i(\widetilde{D_1}) < v_i(D_2)$ for some $E_i$.
\end{itemize}

\end{proposition}

Then we get the following generalization of \cite[Corollary
3.4]{ACAMDC13} to the setting of mixed multiplier ideals.

\vskip 2mm

\begin{corollary} \label{cor:semicontMMI}
Let $\fab:=\left(\fa_1,\dots,\fa_r \right)\subseteq (\cO_{X,O})^r$ be a tuple
of ideals and $\lambdab, \lambdapb\in\bR^r_{\geqslant 0}$.
Let
$D_{\lambdab}=\sum e^{\lambdab}_j E_j$ be the antinef closure of
$\left\lfloor \lambda_1 F_1 + \cdots + \lambda_r F_r  - K_{\pi}
\right\rfloor$. Then:

\begin{itemize}

\item[] $\lambdapb \in \cR_{\fab}(\lambdab) $
if and only if $\lfloor \lambda'_1 e_{1,j} + \cdots +\lambda'_r
e_{r,j} - k_j \rfloor \leqslant e^{\lambdab}_j$ for all $E_j$,

\end{itemize}

\end{corollary}

\vskip 2mm

With the technical tools stated above we are ready for the main
result of this section. Namely, we provide a formula to compute the
region associated to any given point that is a generalization of
\cite[Theorem 3.5]{ACAMDC13} in the context of mixed multiplier
ideals.

\begin{theorem} \label{thm:region}
Let $\fab:=\left(\fa_1,\dots,\fa_r \right)\subseteq \left(\Oc_{X,O}\right)^r$ be a tuple of
ideals and let $D_{{\lambdab}}=\sum e_j^{{\lambdab}} E_j$ be the
antinef closure of $\left\lfloor \lambda_1 F_1 + \cdots +
\lambda_rF_r  - K_{\pi} \right\rfloor$ for a given $ {\lambdab}\in
\bR^r_{\geqslant 0}$. Then the region of $\lambdab$ is the rational
convex polytope determined by the inequalities
\[e_{1,j} z_1 + \cdots +  e_{r,j} z_r < k_j + 1 + e_j^{{\lambdab}}\,,\]
corresponding to either rupture or dicritical divisors $E_j$.

\end{theorem}

\vskip 2mm

In order to prove the second part of this result, we need to invoke some
results on {\it jumping divisors} that will be develop in Section 
\ref{jumping_divisor}.

\begin{proof}
It follows from Corollary \ref{cor:semicontMMI} that
 $\lambdapb$ is not in the region if and only if there exists $E_j$ such that
\[\lfloor \lambda_1' e_{1,j}+\cdots+\lambda_r' e_{r,j} - k_j \rfloor > e^{\lambdab}_j.\]
This inequality is equivalent to, $-k_j+ \lambda_1' e_{1,j}+\cdots+\lambda_r' e_{r,j} \geqslant e_j^{{\lambdab}}+1$
and therefore to $\lambda_1' e_{1,j}+\cdots+\lambda_r' e_{r,j}\geqslant
k_j+1+e_j^{\lambdab}$.

\vskip 2mm

To finish the proof, we have to prove that we only need to consider the rupture
or dicritical divisors. Let $H_j:  e_{1,j} z_1+ \cdots +  e_{r,j} z_r = k_j + 1 +
e_j^{{\lambdab}} $ be the hyperplane associated to the divisor $E_j$ considered above.
Then, among all the exceptional divisors $E_i$ such that $e_{1,i} z_1+ \cdots +  e_{r,i} z_r = k_i + 1 +
e_j^{{\lambdab}} $ gives the same hyperplane $H$, we may find a rupture or dicritical divisor
by Theorem \ref{thm:geo_dual_graphMMI}.
\end{proof}

\begin{remark}

When $X$ has a rational singularity at $O$, we may have a strict
inclusion $\Oc_{X,O}\varsupsetneq\J(\fab^{\lambdab_0})$ for
$\lambdab_0=(0,\dots,0)$. The above result for this case gives a
mild generalization of the well-known formula for the region  $\cR_{\fab}(\lambdab_0)$
in the smooth case (see \cite{LM11} where this region is denoted LCT-polytope). Namely, it is the rational
convex polytope determined by the inequalities
\[e_{1,j} z_1 + \cdots +  e_{r,j} z_r < k_j + 1 + e_j^{{\lambdab_0}}\,,\]
corresponding to either rupture or dicritical divisors $E_j$.

\end{remark}

\begin{remark}
When $X$ is smooth, Cassou-Nogu\`es and Libgober \cite{CNL11,CNL14} studied the analogous
notions of ideals and faces of
quasi-adjunction of a tuple of irreducible plane curves. In particular, they obtained a formula for
the region associated to any given germ $\phi \in \cO_{X,O}$ that resembles
the one given in Theorem \ref{thm:region} (see \cite[Proposition 2.2]{CNL11} and \cite[Theorem 4.1]{CNL14}).
\end{remark}

\begin{corollary}
Let $\fab:=\left(\fa_1,\dots,\fa_r \right)\subseteq (\cO_{X,O})^r$ be a tuple
of ideals. Then the region $\cR_{\fab}(\lambdab)$ is bounded for any point $\lambdab \in \bR^r_{\geqslant 0}$.

\end{corollary}

This property enables us to give a recursive way to compute the constancy region
 $\cC_{\fab}(\lambdab)$ from the finitely many constancy regions satisfying
 $\cC_{\fab}(\lambdapb) \preccurlyeq \cC_{\fab}(\lambdab)$.

\begin{corollary}\label{formula_constancy}
 Let $\fab:=\left(\fa_1,\dots,\fa_r \right)\subseteq (\cO_{X,O})^r$ be a tuple
of ideals. Given $\lambdab \in \bR^r_{\geqslant 0}$, there exists finitely many 
points $\lambdab_1,\dots,\lambdab_k \in \bR^r_{\geqslant 0}$ such that
\begin{equation} \label{eq:const_region3}
\cC_{\fab}(\lambdab)=\cR_{\fab}(\lambdab) \backslash
\left(\cR_{\fab}(\lambdab_1)\cup \cdots \cup \cR_{\fab}(\lambdab_{k})\right)=\cR_{\fab}(\lambdab) \backslash
\left(\cC_{\fab}(\lambdab_1)\cup \cdots \cup \cC_{\fab}(\lambdab_{k})\right).
\end{equation}
In particular, $\cC_{\fab}(\lambdab_1), \dots, \cC_{\fab}(\lambdab_{k})$
are all the constancy regions that are strictly smaller than $\cC_{\fab}(\lambdab)$ using
the partial order $\preccurlyeq$.
\end{corollary}

\begin{remark}
 To obtain a simpler expression in the first equation of (\ref{eq:const_region3}) we may choose
$\lambdab_1, \dots , \lambdab_s \in \bR^r_{\geqslant 0}$ such that $\cC_{\fab}(\lambdab_1), \dots ,
\cC_{\fab}(\lambdab_s)$ are the maximal elements among those constancy regions which are strictly smaller
than $\cC_{\fab}(\lambdab)$ using the partial order $\preccurlyeq$. Then
\begin{equation}  \label{eq:const_region4}
\cC_{\fab}(\lambdab)=\cR_{\fab}(\lambdab) \backslash
\left(\cR_{\fab}(\lambdab_1)\cup \cdots \cup \cR_{\fab}(\lambdab_{s})\right).
\end{equation}

\end{remark}

\vskip 2mm

Theorem \ref{thm:region} is one of the key ingredients for the algorithm
that we will present in Section \ref{ssec:algorithm}. The other key ingredient
comes from a careful study of the $\cC$-facets of the components of a
constancy regions that will show their subtlety.

\vskip 2mm

For simplicity, due to the fact that for a fixed jumping point
$\lambdab$, for $\varepsilon>0$ sufficiently
small any ${\lambdapb \in  \{\lambdab - \R_{\geqslant0}^r\}
\cap B_{\varepsilon}(\lambdab)}$  defines the same mixed multiplier ideal, we will denote this
mixed multiplier ideal as the one associated to
$(1-\varepsilon)\lambdab$ for $\varepsilon>0$ sufficiently small.

\vskip 2mm

We start with the following well-known fact.

\begin{lemma}
 Let $\fab:=\left(\fa_1,\dots,\fa_r \right)\subseteq (\cO_{X,O})^r$ be a tuple
of ideals and $\lambdab \in \bR^r_{\geqslant 0}$ be a point.

\begin{enumerate}
 \item The interior of a $\cC$-facet, as a subspace of its supporting hyperplane, is non-empty.

 \item Any constancy region $\cC_{\fab}(\lambdab)$ different from the constancy region associated to the origin,
 has non-empty intersection with the interior of some $\cC$-facets.

 \item Any interior point $\lambdapb$ of a $\cC$-facet of $\cC_{\fab}(\lambdab)$ satisfies
  $\J(\fab^{(1-\varepsilon)\lambdapb}) = \J(\fab^{\lambdab})$
\end{enumerate}

\end{lemma}

\begin{proof}
The key point in the proof of these three statements is that, for all $\varepsilon >0$, we have that
$B_{\varepsilon}(\lambdab) \cap \cR_{\fab}(\lambdab)$ contains an open ball $B_{\varepsilon}(\mub)$
for some $\mub \in \cR_{\fab}(\lambdab)$. To finish the proof of ii) we notice that the
inner boundary $\cC_{\fab}(\lambdab) \backslash \cC_{\fab}(\lambdab)^\circ$ provides the points
of $\cC_{\fab}(\lambdab)$ which are interior points of a $\cC$-facet of some other constancy region, which
is necessarily smaller than $\cC_{\fab}(\lambdab)$ using the partial order $\preccurlyeq$.
\end{proof}

The key result states that a $\cC$-facet cannot {\it be crossed} by any jumping wall.

\begin{proposition}\label{inside_point}

Let $\fab:=\left(\fa_1,\dots,\fa_r \right)\subseteq (\cO_{X,O})^r$ be a tuple
of ideals. Let $\lambdab$ and $\lambdapb$ be interior points of the same
$\cC$-facet of a constancy region. Then we have
$\J\left(\fab^{\lambdab}\right)=\J\left(\fab^{\lambdapb}\right)$.

\end{proposition}

Once again we need to use some results from  Section 
\ref{jumping_divisor} to prove this fact.

\begin{proof}
Let $H$ be the supporting hyperplane of the $\cC$-facet containing $\lambdab$ and $\lambdapb$.
Notice that both are jumping points coming from the same mixed multiplier ideal, i.e.,
${\J\left(\fab^{(1-\varepsilon)\lambdab}\right)= \J\left(\fab^{(1-\varepsilon)\lambdapb}\right)}$.
For simplicity we take a point $\mub$ as a representative of the constancy region of this ideal.
Now, let $D_{\mub}= \sum e_j^{\mub} E_j$ be the antinef closure of
$\lfloor \mu_1 F_1 + \cdots + \mu_r F_r - K_\pi  \rfloor$. Consider the reduced divisor $G$
supported on those exceptional components $E_j$ such that the hyperplane $H$ has equation
$$ \lambda_1e_{1,j}+ \cdots +   \lambda_r e_{r,j} = k_j + 1 +
e_j^{{\mub}}\,. $$
Then, using Lemma \ref{lem:Constancy_facet} and Proposition \ref{prop:jump2_MMI} we have
$$\J\left(\fab^{\lambdab}\right)= \pi_{*}\Oc_{X'}(-D_{(1-\varepsilon)\lambdab}-G) =\J\left(\fab^{\lambdapb}\right).$$
\end{proof}





\subsection{An algorithm to compute the constancy regions}\label{ssec:algorithm}
The algorithm that we are going to present is a generalization of
the one given in \cite[Algorithm 3.8]{ACAMDC13} that we briefly
recall. Given an ideal  $\fa \subseteq \cO_{X,O}$, we construct
sequentially the chain of multiplier ideals
$$\Oc_{X,O}\supseteq\J(\A^{\lambda_0})\varsupsetneq\J(\A^{\lambda_1})\varsupsetneq \J(\A^{\lambda_2})\varsupsetneq...\varsupsetneq\J(\A^{\lambda_i})\varsupsetneq...$$
The starting point is to compute the multiplier ideal associated to
$\lambda_0=0$  by means of the antinef closure $D_{\lambda_0}=\sum
e_i^{\lambda_0} E_i$ of $\lfloor - K_\pi \rfloor$ using the
unloading procedure described in Section \ref{unloading}. The
log-canonical threshold is
\begin{equation*} 
\lambda_1=  \min_i\left\{\frac{k_i+1 +
e_i^{\lambda_0}}{e_i}\right\}\,.
\end{equation*}
so we may  describe its associated multiplier ideal
$\J(\A^{\lambda_1})$ just computing the antinef closure
$D_{\lambda_1}=\sum e_i^{\lambda_1} E_i$ of $\lfloor \lambda_1 F -
K_\pi \rfloor$ using the unloading procedure. By \cite[Theorem
3.5]{ACAMDC13},  the next jumping number is
\begin{equation*} 
\lambda_2=  \min_i\left\{\frac{k_i+1 +
e_i^{\lambda_1}}{e_i}\right\}.
\end{equation*}
 Then we only have to
follow the same strategy: the antinef closure $D_{\lambda_2}$ of
$\lfloor \lambda_2 F - K_\pi \rfloor$, i.e., the multiplier ideal
$\J(\A^{\lambda_2})$,  allows us to compute $\lambda_3$ and so on.

\vskip 2mm

We may interpret that at each step of the algorithm, the jumping
number $\lambda_i$ allows us to compute its region, and equivalently
its constancy region $[\lambda_i,\lambda_{i+1})$. The boundary of
this constancy region gives us the next jumping number
$\lambda_{i+1}$. In particular we have a one-to-one correspondence
between the constancy regions of the ideal $\fa$ and the jumping
numbers.

\vskip 2mm

The algorithm  for mixed multiplier ideals is more involved. It
starts with the computation of the mixed multiplier ideal associated
to $\lambdab_0=(0,\dots,0)$, using the unloading procedure. The
region $\cR_{\fab}(\lambdab_0)$  is described by means of the
formula given in Theorem \ref{thm:region}. In this case the region
coincides with the constancy region $\cC_{\fab}(\lambdab_0)$, so we
have a nice description of its boundary. For  each $\cC$-facet,
using Proposition \ref{inside_point}, we may take a single point as
a representative. The next step of the algorithm is to compute the
mixed multiplier ideals of these points in order to describe their
corresponding regions, using Theorem \ref{thm:region} once again.
Then we compute the corresponding constancy regions and their
$\cC$-facets and we  follow the same strategy.

\vskip 2mm

Roughly speaking, our strategy is to consider a discrete set of points
comprising one interior point of each $\cC$-facet. This gives a surjective
correspondence with the partially ordered set of constancy regions.
This correspondence is far from being one-to-one as in the case of a single ideal.
To keep track of these points we will consider two sets $N$ and $D$. $N$
will contain the points for which we still have to compute the
corresponding region and, once this region has been computed, we
move it to $D$. In particular, we will start with $N=\{\lambdab_0\}$
and $D=\emptyset$  the empty set.

\vskip 2mm

\begin{algorithm} { (Constancy regions and  mixed multiplier ideals)} \label{A2}

\vskip 2mm

\noindent {\tt Input:} a common log-resolution of the tuple of
 ideals $\fab=\left(\fa_1,\dots , \fa_r \right)\subseteq (\cO_{X,O})^r$. 

\noindent {\tt Output:} list of constancy regions of $\fab$ and their
corresponding mixed multiplier ideals. 

 \vskip 2mm



 Set $N=\{\lambdab_0= (0,\dots,0)\}$ and $D=\emptyset$. From $j=0$ , incrementing by $1$

\vskip 2mm

\begin{itemize}

\item[\textbf{(Step j)}]:

\vskip 2mm

 \begin{enumerate}
  \item[$(j.1)$] {\bf Choosing a convenient point in the set $N$}:

        \begin{itemize}
           \item[$\cdot$] Pick  $\lambdab_j$ the first point in the set
           $N$ and compute its region $\cR_{\fab}(\lambdab_j)$ using Theorem
           \ref{thm:region}.

           \item[$\cdot$] If there is some $\lambdab \in N$ such that
           $\lambdab \in \cR_{\fab}(\lambdab_j)$ and
           $\J(\A^{\lambdab}) \neq \J(\A^{\lambdab_j})$  then put
           $\lambdab$ first in the list $N$ and repeat this step
           $(j.1)$. Otherwise continue with step $(j.2)$.

        \end{itemize}

       \vskip 2mm

  \item [$(j.2)$] {\bf Checking out whether the region has been already computed}:

\begin{itemize}
           \item[$\cdot$] If some $\lambdab \in D$ satisfies $\J(\A^{\lambdab}) = \J(\A^{\lambdab_j})$
           then go to step $(j.4)$. Otherwise continue with step $(j.3)$.

\end{itemize}

         \vskip 2mm

  \item [$(j.3)$] {\bf Picking new points for which we have to compute its region}:

  \begin{itemize}
          \item[$\cdot$] Compute
$$
\cC(j)=\cR_{\fab}(\lambdab_j) \backslash
\left(\cR_{\fab}(\lambdab_1)\cup \cdots \cup \cR_{\fab}(\lambdab_{j-1})\right).
$$

           \item[$\cdot$] For each connected component of
           $\cC(j)$ compute its outer facets\footnote{The outer facets of $\cC(j)$ are the intersection of the boundary
           of any connected component of $\cC(j)$ with a supporting hyperplane of $\cR_{\fab}(\lambdab_j)$.}.

           \item[$\cdot$] Pick one interior point in each
           outer facet of $\cC(j)$ and add them as the last point in $N$.

   \end{itemize}

         \vskip 2mm

  \item [$(j.4)$] {\bf Update the sets $N$ and $D$}:

  \begin{itemize}
           \item[$\cdot$] Delete $\lambdab_j$ from $N$ and add
           $\lambdab_j$ as the last point in $D$.

  \end{itemize}

 \end{enumerate}

 \end{itemize}

\end{algorithm}

\vskip 2mm

\begin{remark}
 Several points of the algorithm require a comparison between mixed multiplier ideals
 (an inequality in step $(j.1)$ and an equality in step $(j.2)$).
 This can be done computing antinef closures of divisors using the unloading procedure.
 For the computation of the region $\cR_{\fab}(\lambdab)$ (steps $(j.1)$ and $(j.3)$) we use Theorem \ref{thm:region}.

\end{remark}

\begin{remark} \label{rem: minimal_j.1}
Step $(j.1)$ is equivalent to choosing a point whose constancy region is a minimal element by the
order $\preccurlyeq$ among those associated to the points in the set $N$.
Any finite subset endowed with a partial ordering has some minimal element, thus
there exists a convenient point in the set $N$ that allows to continue with step $(j.2)$.
\end{remark}

\begin{lemma}
At each step $j$, the algorithm overcomes step $(j.1)$ and provides updated sets $N$ and $D$.
\end{lemma}

\begin{theorem} \label{region_Cj}
The constancy region of the point $\lambdab_j$ chosen
at step $(j.1)$ is computed at step $(j.3)$ of the algorithm, i.e., $\cC(j)=\cC_{\fab}(\lambdab_j)$,
and one interior point for each $\cC$-facet of $\cC_{\fab}(\lambdab_j)$ is added to the set $N$.
\end{theorem}

\begin{proof}
We argue by induction on $j$. For the case $j=1$ the statement holds since we pick $\lambdab_1=\lambdab_0$
at step $(1.1)$ and step $(1.3)$ is performed.

\vskip 2mm

Now assume that the statement is true all the steps up to $j-1$. We want to prove it for step $j$.
Without loss of generality we may assume that step $(j.3)$ must be performed, so
$\J(\A^{\lambdab_i}) \neq \J(\A^{\lambdab_j})$ for all $1\leqslant i \leqslant j-1$. Notice that, by equation
(\ref{eq:const_region4}),  $\cC(j)=\cC_{\fab}(\lambdab_j)$ is equivalent to the fulfillment of the following
two conditions:

\begin{itemize}

\item[a)] Each $\lambdab_i$, $1\leqslant i \leqslant j-1$, satisfies either
$\cC_{\fab}(\lambdab_i) \preccurlyeq \cC_{\fab}(\lambdab_j)$ or both constancy regions are not related by the partial order.

\item[b)] Consider a set  $\{\mub_1,\dots , \mub_s\} \subset \bR_{\geqslant 0}^r$ of representatives of the
constancy regions which are maximal elements among those constancy regions smaller than $\cC_{\fab}(\lambdab_j)$.
Then, for each $k\in \{1,\dots,s\}$ there is some $i_k \in \{1,\dots,j-1\}$ such that
$\cC_{\fab}(\lambdab_{i_k}) \preccurlyeq \cC_{\fab}(\mub_k)$.

\end{itemize}

First we are going to prove that condition a) is satisfied. Assume the contrary, i.e. there exists $i<j$
with $\cC_{\fab}(\lambdab_i) \succ \cC_{\fab}(\lambdab_j)$, that is
$\cR_{\fab}(\lambdab_i)\varsupsetneq \cR_{\fab}(\lambdab_j)$. Assume that $\lambdab_j$ was added to $N$
at step $m <j$. Hence, by induction hypothesis $\lambdab_j$ is an interior point of some $\cC$-facet
of $\cR_{\fab}(\lambdab_m)$, and in particular $\cR_{\fab}(\lambdab_m)\subseteq \cR_{\fab}(\lambdab_j)$.
Thus $\cR_{\fab}(\lambdab_m) \varsubsetneq \cR_{\fab}(\lambdab_i)$, i.e. $\cC_{\fab}(\lambdab_m) \prec \cC_{\fab}(\lambdab_i)$.
We distinguish two cases:

\begin{itemize}
\item[$\cdot$] If $i<m$ we get a contradiction with the induction hypothesis at step $m$ since
condition a) is not fulfilled.

\item[$\cdot$] If $i>m$, we have that $\lambdab_j$ already belongs to $N$ at step $i$. This contradicts the
requirement of step $(i.1)$ which says that $\lambdab_j$ should be treated before $\lambdab_i$.

\end{itemize}

Finally we prove condition b). Assume the contrary, i.e there exists $\mub_i$ whose constancy region is not
dominated by any $\cC_{\fab}(\lambdab_k)$, $1\leqslant k \leqslant j-1$. Without loss of generality we may assume
that the segment $\overline{\lambdab_0 \mub_i}$ intersect the jumping walls at interior points of the $\cC$-facets,
namely in the jumping points $\lambdab_0= \nub_1, \nub_2,\dots, \nub_m= \mub_i$ with
$\nub_k \in \overline{\cC_{\fab}(\nub_{k-1})}$, and thus $\nub_{k-1} \in \cR_{\fab}(\nub_{k})$.

\vskip 2mm

By induction hypothesis, representatives of each constancy region
$\{\cC_{\fab}(\nub_{1}),\dots, \cC_{\fab}(\nub_{m'}) \}$, $m'<m$, are added to $N$ at some steps before
step $j$, being $\lambdapb$ the last representative. Hence, we still have $\lambdapb \in N$ at step $j$
and $$ \cR_{\fab}(\lambdapb)=\cR_{\fab}(\nub_{m'}) \subseteq \cR_{\fab}(\mub_i)\subsetneq \cR_{\fab}(\lambdab_j).$$
 This contradicts the requirement of step $(j.1)$ for $\lambdab_j$.
\end{proof}

As a consequence of Theorem \ref{region_Cj} we obtain the following

\begin{corollary} \label{cor:setD}
At step $j$ of the algorithm, we have that:

\begin{itemize}
\item[i)] The set $D$ contains at least a representative of each constancy region inside $\cR_{\fab}(\lambdab_j)$.

\item[ii)] The set $D$ contains  a representative of all $\cC$-facets inside $\cR_{\fab}(\lambdab_j)$.

\item[iii)] A complete description of the jumping walls inside $\cR_{\fab}(\lambdab_j)$ is obtained by intersecting
the region $\cR_{\fab}(\lambdab_j)$ with the jumping walls associated to the points $\lambdab_1,\dots , \lambdab_{j-1}$.

\end{itemize}

\end{corollary}

\begin{proof}
 From the proof of Theorem \ref{region_Cj} we infer that at step $j$, the maximal elements among
 all the constancy regions inside $\cR_{\fab}(\lambdab_j)$ have already representants
 $\lambdab_{i_1},\dots , \lambdab_{i_s}$ in $D$, $i_1< \cdots < i_s < j$. Arguing by reverse induction with any
 of these points  $\lambdab_{i_k}$, the first claim follows.

 \vskip 2mm

 Now, the statement of Theorem \ref{region_Cj} asserts that at each step $i$ of the algorithm, a representative
 of each $\cC$-facet of $\cC_{\fab}(\lambdab_i)$ is added to $N$. If we only take into account the points $\lambdab_i$
 of constancy regions inside $\cR_{\fab}(\lambdab_j)$, the subsequent representatives in $\cC$-facets still lying inside
 $\cR_{\fab}(\lambdab_j)$ must be treated (and added to $D$) before $\lambdab_j$, in virtue of step $(j.1)$ of the
 algorithm.

 \vskip 2mm

 Part iii) of the statement is a direct consequence of claim i).
\end{proof}

\begin{remark} \label{rem:setD}
 Each point $\lambdab$ included in $N$ at some step of the algorithm  is treated
after a finite number of steps and added to $D$. Indeed, the order of incorporation of the points in $N$ is
preserved unless step $(j.1)$ priorizes some other point. This happens only a finite amount of times
since there is only a finite number of constancy regions inside any given region.

\end{remark}

\begin{proposition}
Once a point $\lambdab \in \bR_{\geqslant 0}^r$ is fixed, a set $D$ which includes a representative of all
constancy regions in the compact $(\lambdab_0 + \bR_{\geqslant 0}^r) \cap (\lambdab - \bR_{\geqslant 0}^r)$
is achieved after a finite number of steps of the algorithm.

\end{proposition}

\begin{proof}
 Observe that $(\lambdab_0 + \bR_{\geqslant 0}^r) \cap (\lambdab - \bR_{\geqslant 0}^r) \subseteq \cR_{\fab}(\lambdab)$.
 In virtue of Corollary \ref{cor:setD} and Remark \ref{rem:setD}, we only have to prove that some representative
 of $\cC_{\fab}(\lambdab)$ is added to $N$ at some step. We may take $\lambdapb \in \cC_{\fab}(\lambdab)$ such that
 the segment $\overline{\lambdab_0 \lambdapb }$ intersects the jumping walls at interior points of $\cC$-facets,
 namely in the jumping points $\lambdab_0= \nub_1, \nub_2,\dots, \nub_m= \lambdapb$. The algorithm starts
 with $\nub_1$ and incorporates $\nub_2$ to $N$. Since $\nub_k \in \overline{\cC_{\fab}(\nub_{k-1})}$, once
 $\nub_k$ is selected at some finite step $i_k$, $\nub_{k+1}$ is added to $N$ at this same step.
 Hence, $\lambdapb$ is selected at some step $(j.1)$. Notice that this implies that no point in $N$
 lies in $\cR_{\fab}(\lambdapb) = \cR_{\fab}(\lambdab)$, i.e. $N \cap \cR_{\fab}(\lambdab)=\emptyset$.

 \vskip 2mm

 Conversely, if at some step $j$ $N \cap \cR_{\fab}(\lambdab)=\emptyset$,  then the new $N$ obtained at any
 forthcoming step still satisfies $N \cap \cR_{\fab}(\lambdab)=\emptyset$. If some
 $\lambdab_i \not \in \cR_{\fab}(\lambdab)$ with $i> j$ is chosen at step $(i.1)$, any new point $\mub$
 added to $N$ at step $(i.3)$ satisfies
 $\J(\fab^{\mub}) \varsubsetneq \J(\fab^{\lambdab_i})  \nsupseteq \J(\fab^{\lambdab})$ and hence
 $\J(\fab^{\mub}) \nsupseteq \J(\fab^{\lambdab})$, equivalently $\mub \not \in \cR_{\fab}(\lambdab)$.
 Since the algorithm starts with $\lambdab_1=\lambdab_0 \in \cR_{\fab}(\lambdab)$, we may conclude that at
 a step where $N \cap \cR_{\fab}(\lambdab)=\emptyset$ necessarily the set $D$ obtained at that step contains
 a representative of $\cR_{\fab}(\lambdab)$.
\end{proof}

\vskip 2mm

We present the following simple example to highlight the nuances of
the procedure. In the example, step $(j.1)$ is
performed when computing the region associated to the point
$\lambdab_5$ and step $(j.2)$ is performed for the points
$\lambdab_2$, $\lambdab_4$, $\lambdab_7$ and $\lambdab_8$. In particular, step  $(j.2)$ is included
to avoid too many computations.

\vskip 2mm

\begin{example}\label{ex:algMMI}

Consider the following set of ideals $\fab=(\fa_1,\fa_2)$ with
$\fa_1=(x^3,y^7)$ and $\fa_2=(x,y^2)$ on a smooth surface $X$. We represent the relative
canonical divisor $K_\pi$ and $F_1$ and $F_2$ in the dual graph as
follows: \hspace{1mm}

\begin{center}

\begin{tabular}{ccc}
   \begin{tikzpicture}[scale=0.9]
   \draw  (-4,0) -- (0,0);
   \draw [dashed,->,thick] (-2,0) -- (-1,1);
   \draw [dashed,->,thick] (-3,0) -- (-2,1);
   \draw (-4.2,-0.3) node {{\small $E_1$}};
   \draw (-3.2,-0.3) node {{\small $E_2$}};
   \draw (-0.2,-0.3) node {{\small $E_3$}};
   \draw (-1.2,-0.3) node {{\small $E_4$}};
   \draw (-2.2,-0.3) node {{\small $E_5$}};
   \filldraw  (0,0) circle (2pt)
              (-1,0) circle (2pt)
              (-4,0) circle (2pt);
   \filldraw  [fill=white]  (-2,0) circle (3pt)
                            (-3,0) circle (3pt);
   \end{tikzpicture} &
\hspace{5mm}
\begin{tikzpicture}[scale=0.9]
   \draw  (-4,0) -- (0,0);
   \draw [dashed,->,thick] (-2,0) -- (-1,1);
   \draw [dashed,->,thick] (-3,0) -- (-2,1);
   \draw (-4.2,-0.3) node {{\small $1$}};
   \draw (-3.2,-0.3) node {{\small $2$}};
   \draw (-0.2,-0.3) node {{\small $3$}};
   \draw (-1.2,-0.3) node {{\small $6$}};
   \draw (-2.2,-0.3) node {{\small $9$}};
   \filldraw  (0,0) circle (2pt)
              (-1,0) circle (2pt)
              (-4,0) circle (2pt);
   \filldraw  [fill=white]  (-2,0) circle (3pt)
                            (-3,0) circle (3pt);
 \end{tikzpicture}&
\hspace{5mm}
\begin{tikzpicture}
   \draw  (-4,0) -- (0,0);
   \draw [dashed,->,thick] (-2,0) -- (-1,1);
   \draw [dashed,->,thick] (-3,0) -- (-2,1);
   \draw (-4.2,-0.3) node {{\tiny  $(3,1)$}};
   \draw (-3.2,-0.3) node {{\tiny   $(6,2)$}};
   \draw (-0.2,-0.3) node {{\tiny   $(7,2)$}};
   \draw (-1.2,-0.3) node {{\tiny   $(14,4)$}};
   \draw (-2.2,-0.3) node {{\tiny  $(21,6)$}};
   \filldraw  (0,0) circle (2pt)
              (-1,0) circle (2pt)
              (-4,0) circle (2pt);
   \filldraw  [fill=white]  (-2,0) circle (3pt)
                            (-3,0) circle (3pt);
 \end{tikzpicture}\\
{ Vertex ordering}&
{ $K_\pi$}&
{ $(F_1, F_2)$}
\end{tabular}

\end{center}

\vskip 2mm The blank dots correspond to dicritical divisors in one
of the ideals and their excesses are represented by broken arrows.
For simplicity we will collect the values of any divisor in a
vector. Namely, we have $K_\pi=(1,2,3,6,9)$, $F_1=(3,6,7,14,21)$ and
$F_2=(1,2,2,4,6)$. In the algorithm we will have to perform several
times unloading steps, so we will have to consider the intersection
matrix $M=(E_i\cdot E_j)_{1\leqslant i,j\leqslant 5}$

{\Small
  \[M=\left( \begin {array}{ccccc} -2&1&0&0&0\\ \noalign{\medskip}1&-4&0&0&
1\\ \noalign{\medskip}0&0&-2&1&0\\ \noalign{\medskip}0&0&1&-2&1
\\ \noalign{\medskip}0&1&0&1&-1
\end {array} \right)\,.\]}

Notice that $E_2$ and $E_5$ are the only dicritical divisors. Then,
as a consequence of Theorem \ref{thm:region}, the region of a given
point $\lambdab=(\lambda_1,\lambda_2)$ is defined by
       $$\begin{array}{lcr}
          6 z_1  +2 z_2   & < &2+1+e_2^{\lambdab}\,,\\
          21 z_1 +6 z_2   & < &9+1+e_5^{\lambdab}\,.
         \end{array}
$$

 We keep track of
what we have to compute with the set $N$ that for the moment will
only contain $\lambdab_0=(0,0)$. The set $D$ that keeps track of
the points that we have already computed will be empty  since we have not
computed anything yet.

\vskip 2mm

$\bullet$ {\bf Step 0}. We start computing the multiplier ideal corresponding to
 $\lambdab_0=(0,0)$. Namely, the antinef closure of the
divisor $\lfloor 0 F_1 + 0 F_2 - K_\pi\rfloor$ is
$D_{\lambdab_0}=0$. The corresponding region
$\Reg_{\fab}(\lambdab_0)$  is given by the inequalities
$$\begin{array}{lcr}
          6 z_1    +2 z_2 & < &3,\\
          21 z_1   +6 z_2 & < &10.
         \end{array}
$$
Notice that the constancy region $\mathcal{C}_{\fab}(\lambdab_0)$
coincides with $\Reg_{\fab}(\lambdab_0)$. Its boundary, i.e. the
corresponding jumping wall, has two $\mathcal{C}$-facets so,
according to Proposition \ref{inside_point}, we only need to consider an interior point of each
$\mathcal{C}$-facet in order to continue our procedure. For
simplicity we consider the barycenters $\left(\frac{1}{6},1\right)$
and $\left(\frac{17}{42},\frac{1}{4}\right)$ corresponding to each
segment.

\vskip 2mm

\begin{itemize}
\item[$\cdot$]
$N=\{\left(\frac{1}{6},1\right),
\left(\frac{17}{42},\frac{1}{4}\right) \}$.
\item[$\cdot$] $D= \{\left(0,0\right) \}$.
\end{itemize}

\vskip 2mm

$\bullet$ {\bf Step 1}. We pick the first point
$\lambdab_1:=\left(\frac{1}{6},1\right)$ in $N$ and we compute its
multiplier ideal. Namely, $\lfloor \frac{1}{6} F_1 +  F_2 -
K_\pi\rfloor=(0,1,0,0,0)$ and its antinef closure is
$D_{\lambdab_1}=(1,1,1,2,3)$, so the region $\Reg_{\fab}(\lambdab_1)$
is given by the inequalities
$$\begin{array}{lcr}
          6 z_1    +2 z_2 & < &4,\\
          21 z_1   +6 z_2 & < &13.
         \end{array}
$$ The constancy region $\mathcal{C}_{\fab}(\lambdab_1)=\Reg_{\fab}(\lambdab_1)\backslash\Reg_{\fab}(\lambdab_0)$
has two $\mathcal{C}$-facets for which we pick the interior points
$\left(\frac{1}{6},\frac{3}{2}\right)$ and
$\left(\frac{10}{21},\frac{1}{2}\right)$ respectively. Then, the
sets $N$ and $D$ are:

\vskip 2mm

\begin{itemize}
\item[$\cdot$]
$N=\{ \left(\frac{17}{42},\frac{1}{4}\right),
\left(\frac{1}{6},\frac{3}{2}\right),
\left(\frac{10}{21},\frac{1}{2}\right) \}$.
\item[$\cdot$] $D= \{\left(0,0\right), \left(\frac{1}{6},1\right) \}$.
\end{itemize}

\vskip 2mm

$\bullet$ {\bf Step 2}. The point
$\lambdab_2:=\left(\frac{17}{42},\frac{1}{4}\right)$ satisfies
$\J(\fab^{\lambdab_2})=\J(\fab^{\lambdab_1})$, so they have the same
region. In order to keep track of all the $\mathcal{C}$-facets we
have to consider this point as well, so the sets $N$ and $D$ that we
get after this step are:

\vskip 2mm
\begin{itemize}
\item[$\cdot$]
$N=\{ \left(\frac{1}{6},\frac{3}{2}\right),
\left(\frac{10}{21},\frac{1}{2}\right) \}$.
\item[$\cdot$] $D= \{\left(0,0\right), \left(\frac{1}{6},1\right), \left(\frac{17}{42},\frac{1}{4}\right)\}$.
\end{itemize}

\vskip 2mm

$\bullet$ {\bf Step 3}.  We pick
$\lambdab_3:=\left(\frac{1}{6},\frac{3}{2}\right)$. We have $\lfloor
\frac{1}{6} F_1 + \frac{3}{2} F_2 - K_\pi\rfloor=(1,2,1,2,3)$ and
its antinef closure is $D_{\lambdab_3}=(1,2,2,4,6)$, so the region
$\Reg_{\fab}(\lambdab_3)$ is given by the inequalities
$$\begin{array}{lcr}
          6 z_1    +2 z_2 & < &5,\\
          21 z_1   +6 z_2 & < &16.
         \end{array}
$$ The constancy region $\mathcal{C}_{\fab}(\lambdab_3)=\Reg_{\fab}(\lambdab_3)\backslash (
\Reg_{\fab}(\lambdab_0) \cup \Reg_{\fab}(\lambdab_1) \cup \Reg_{\fab}(\lambdab_3)) =
\Reg_{\fab}(\lambdab_3)\backslash\Reg_{\fab}(\lambdab_1)$
has two $\mathcal{C}$-facets for which we pick the interior points
$\left(\frac{1}{6},2\right)$ and
$\left(\frac{23}{42},\frac{3}{4}\right)$ respectively. Then, the
sets $N$ and $D$ are:

\vskip 2mm
\begin{itemize}
\item[$\cdot$]
$N=\{  \left(\frac{10}{21},\frac{1}{2}\right),
\left(\frac{1}{6},2\right),\left(\frac{23}{42},\frac{3}{4}\right)
\}$.
\item[$\cdot$] $D= \{\left(0,0\right), \left(\frac{1}{6},1\right), \left(\frac{17}{42},\frac{1}{4}\right), \left(\frac{1}{6},\frac{3}{2}\right)\}$.
\end{itemize}

\vskip 2mm

\begin{figure}[ht!!]
  \begin{center}
  \medskip
  \includegraphics[width=.32\textwidth]{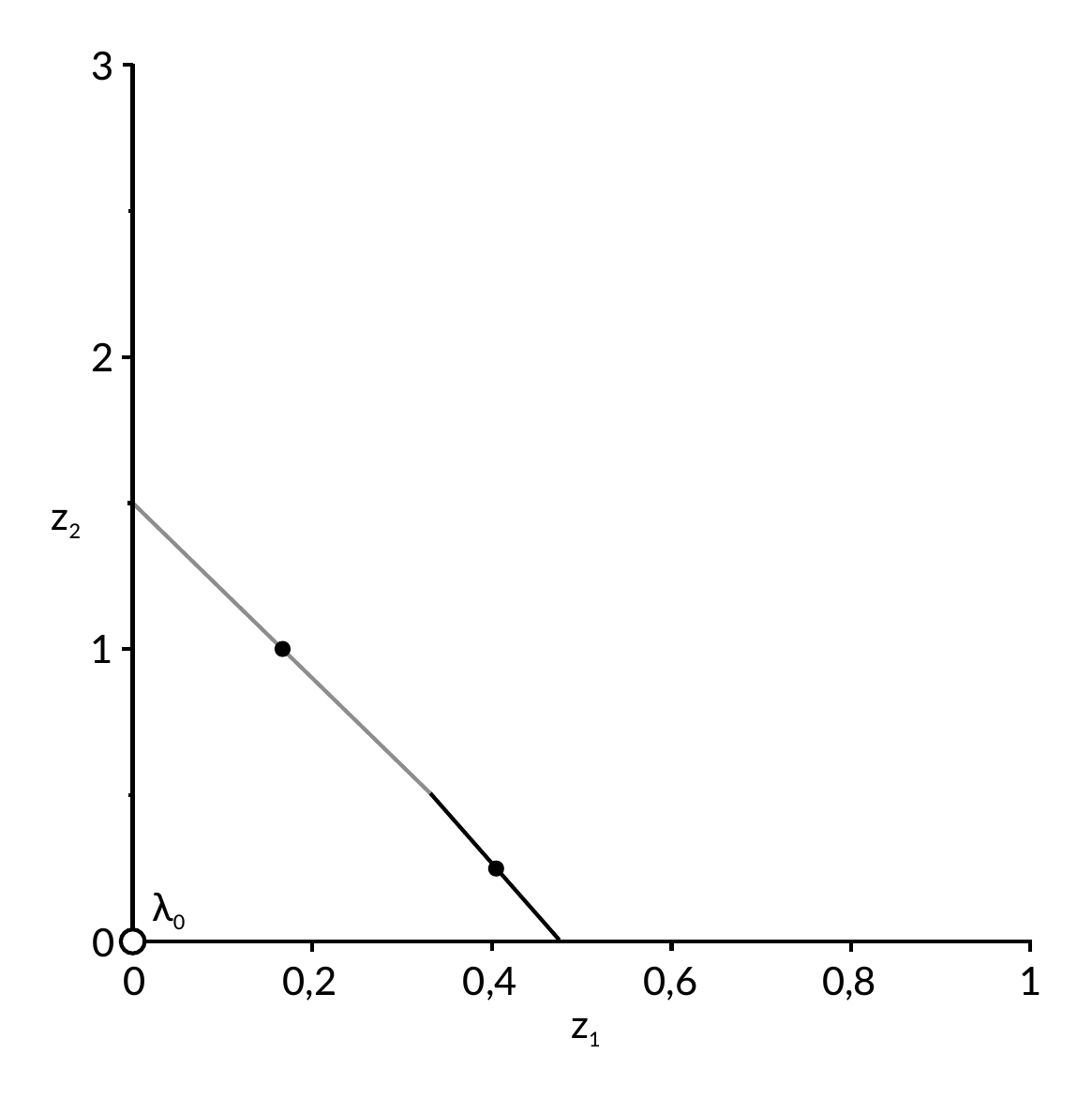}
  \includegraphics[width=.32\textwidth]{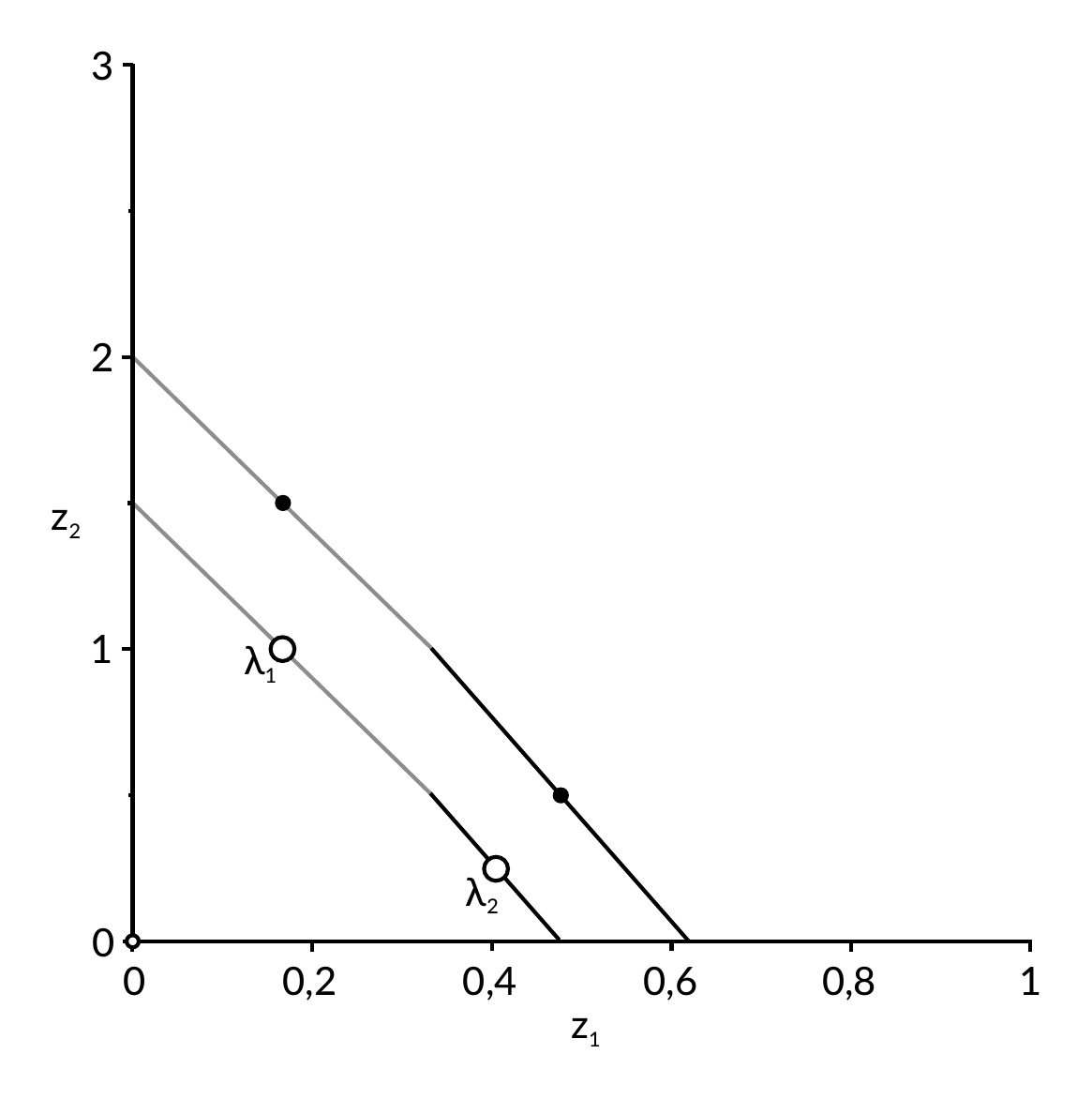}
  \includegraphics[width=.32\textwidth]{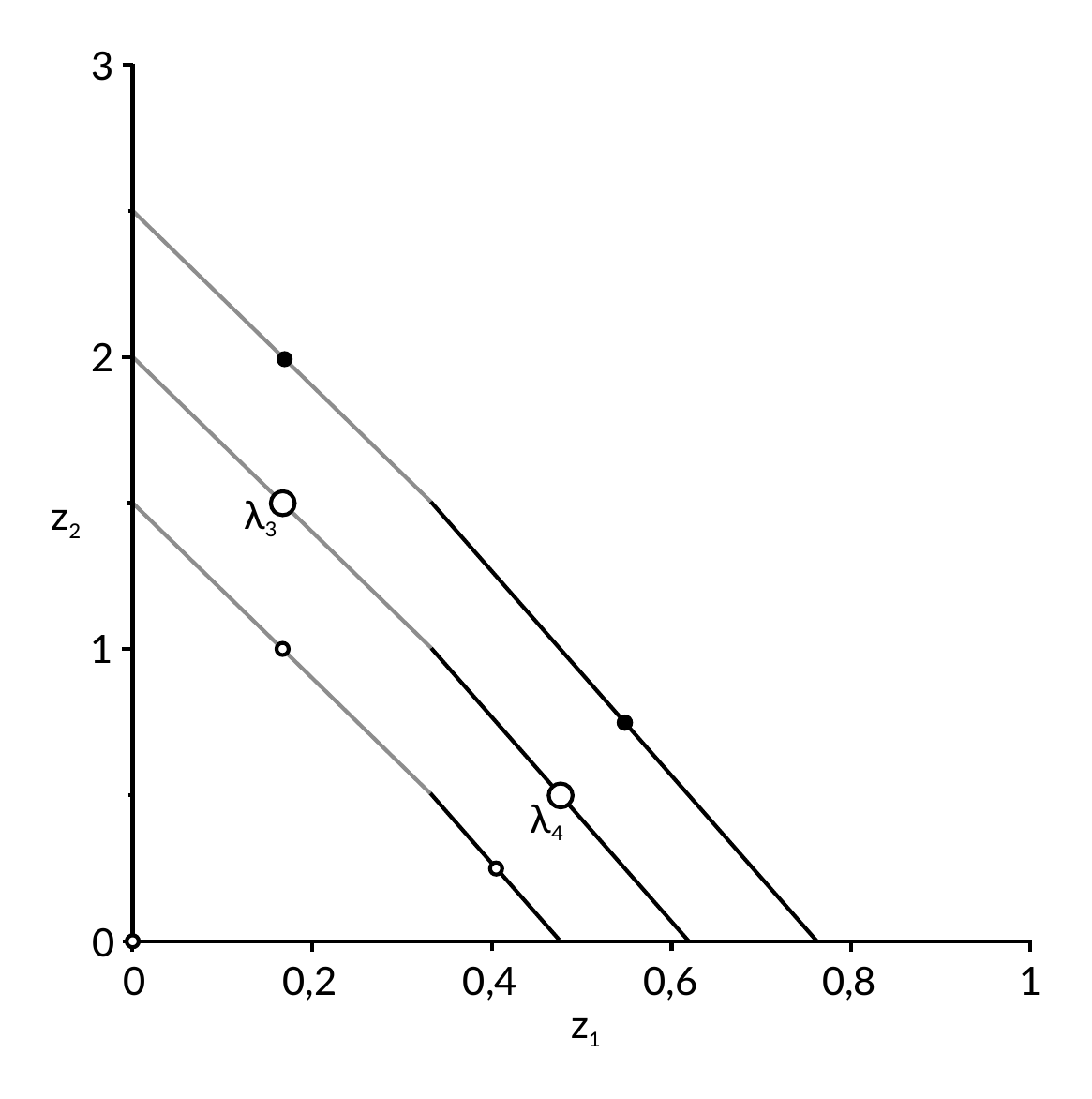}
  \end{center}
\caption{Constancy regions associated to $\lambdab_0$, $\lambdab_1$
(equivalently $\lambdab_2$) and $\lambdab_3$ (equivalently $\lambdab_4$).}
\end{figure}

\vskip 2mm

$\bullet$ {\bf Step 4}. The point
$\lambdab_4:=\left(\frac{10}{21},\frac{1}{2}\right)$ satisfies
$\J(\fab^{\lambdab_4})=\J(\fab^{\lambdab_3})$ so they have the same
region. We update the sets $N$ and $D$ to obtain:

\vskip 2mm
\begin{itemize}
\item[$\cdot$]
$N=\{\left(\frac{1}{6},2\right),\left(\frac{23}{42},\frac{3}{4}\right)
\}$.
\item[$\cdot$] $D= \{\left(0,0\right), \left(\frac{1}{6},1\right), \left(\frac{17}{42},\frac{1}{4}\right),
\left(\frac{1}{6},\frac{3}{2}\right),
\left(\frac{10}{21},\frac{1}{2}\right)\}$.
\end{itemize}

\vskip 2mm

$\bullet$ {\bf Step 5}. We have that the region associated to
$\left(\frac{23}{42},\frac{3}{4}\right)$ is contained  in the region
of $\left(\frac{1}{6},2\right)$. It is for this reason that we will
consider first the point
$\lambdab_5:=\left(\frac{23}{42},\frac{3}{4}\right)$. We have
$\lfloor \frac{23}{42} F_1 + \frac{3}{4} F_2 -
K_\pi\rfloor=(1,2,2,4,7)$ and its antinef closure is
$D_{\lambdab_5}=(1,2,3,5,7)$ so the region $\Reg_{\fab}(\lambdab_3)$
is given by the inequalities
$$\begin{array}{lcr}
          6 z_1    +2 z_2 & < &5,\\
          21 z_1   +6 z_2 & < &17.
         \end{array}
$$ The constancy region $\mathcal{C}_{\fab}(\lambdab_5)=\Reg_{\fab}(\lambdab_5)\backslash\Reg_{\fab}(\lambdab_3)$
has two $\mathcal{C}$-facets for which we pick the interior points
$\left(\frac{1}{2},1\right)$ and
$\left(\frac{31}{42},\frac{1}{4}\right)$ respectively. Then, the
sets $N$ and $D$ are:

\vskip 2mm
\begin{itemize}
\item[$\cdot$]
$N=\{\left(\frac{1}{6},2\right),
\left(\frac{1}{2},1\right),\left(\frac{31}{42},\frac{1}{4}\right)
\}$.
\item[$\cdot$] $D= \{\left(0,0\right), \left(\frac{1}{6},1\right), \left(\frac{17}{42},\frac{1}{4}\right),
\left(\frac{1}{6},\frac{3}{2}\right),
\left(\frac{10}{21},\frac{1}{2}\right),\left(\frac{23}{42},\frac{3}{4}\right)\}$.
\end{itemize}

\vskip 2mm

$\bullet$ {\bf Step 6}. We pick now $\lambdab_6:=\left(\frac{1}{6},2\right)$. We
have $\lfloor \frac{1}{6} F_1 + 2 F_2 - K_\pi\rfloor=(1,3,2,4,6)$
and its antinef closure is $D_{\lambdab_6}=(2,3,3,6,9)$ so the
region $\Reg_{\fab}(\lambdab_6)$ is given by the inequalities
$$\begin{array}{lcr}
          6 z_1    +2 z_2 & < &6,\\
          21 z_1   +6 z_2 & < &19.
         \end{array}
$$ The constancy region $\mathcal{C}_{\fab}(\lambdab_6)=\Reg_{\fab}(\lambdab_6)\backslash\Reg_{\fab}(\lambdab_5)$
has two $\mathcal{C}$-facets for which we pick the interior points
$\left(\frac{1}{6},\frac{5}{2}\right)$ and
$\left(\frac{13}{21},1\right)$. Then, the sets $N$ and $D$ are:

\vskip 2mm
\begin{itemize}
\item[$\cdot$]
$N=\{\left(\frac{1}{2},1\right),\left(\frac{31}{42},\frac{1}{4}\right),
\left(\frac{1}{6},\frac{5}{2}\right)\left(\frac{13}{21},1\right)
\}$.
\item[$\cdot$] $D= \{\left(0,0\right), \left(\frac{1}{6},1\right), \left(\frac{17}{42},\frac{1}{4}\right),
\left(\frac{1}{6},\frac{3}{2}\right),
\left(\frac{10}{21},\frac{1}{2}\right),\left(\frac{23}{42},\frac{3}{4}\right),
\left(\frac{1}{6},2\right)\}$.
\end{itemize}

\vskip 2mm

\begin{figure}[ht!!!!!!!!!!!!!!!!!!!!!!!!!!!!!!!!!!!!!!!!!!]
  \begin{center}
  \medskip
  \includegraphics[width=.3\textwidth]{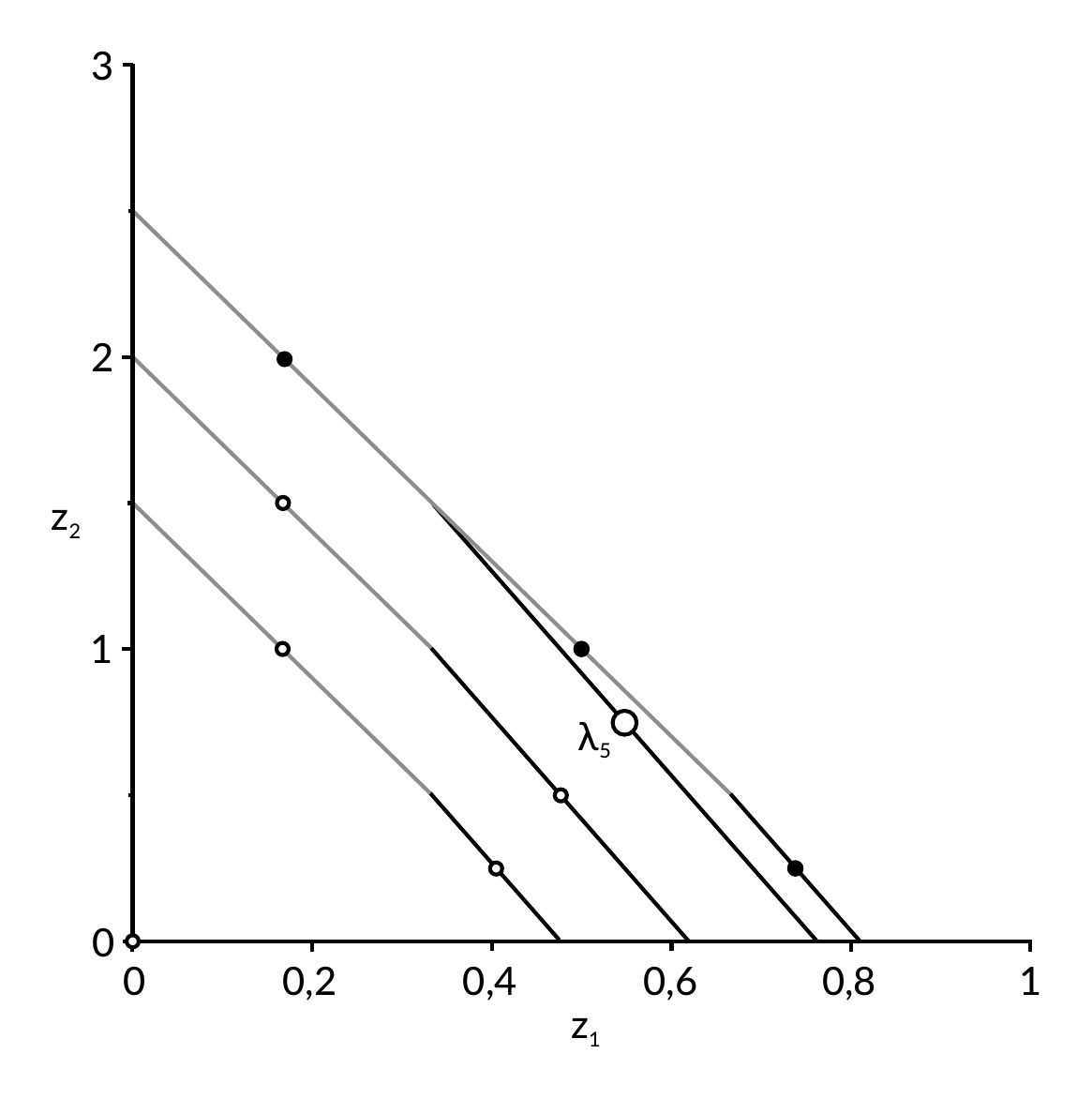}
  \includegraphics[width=.3\textwidth]{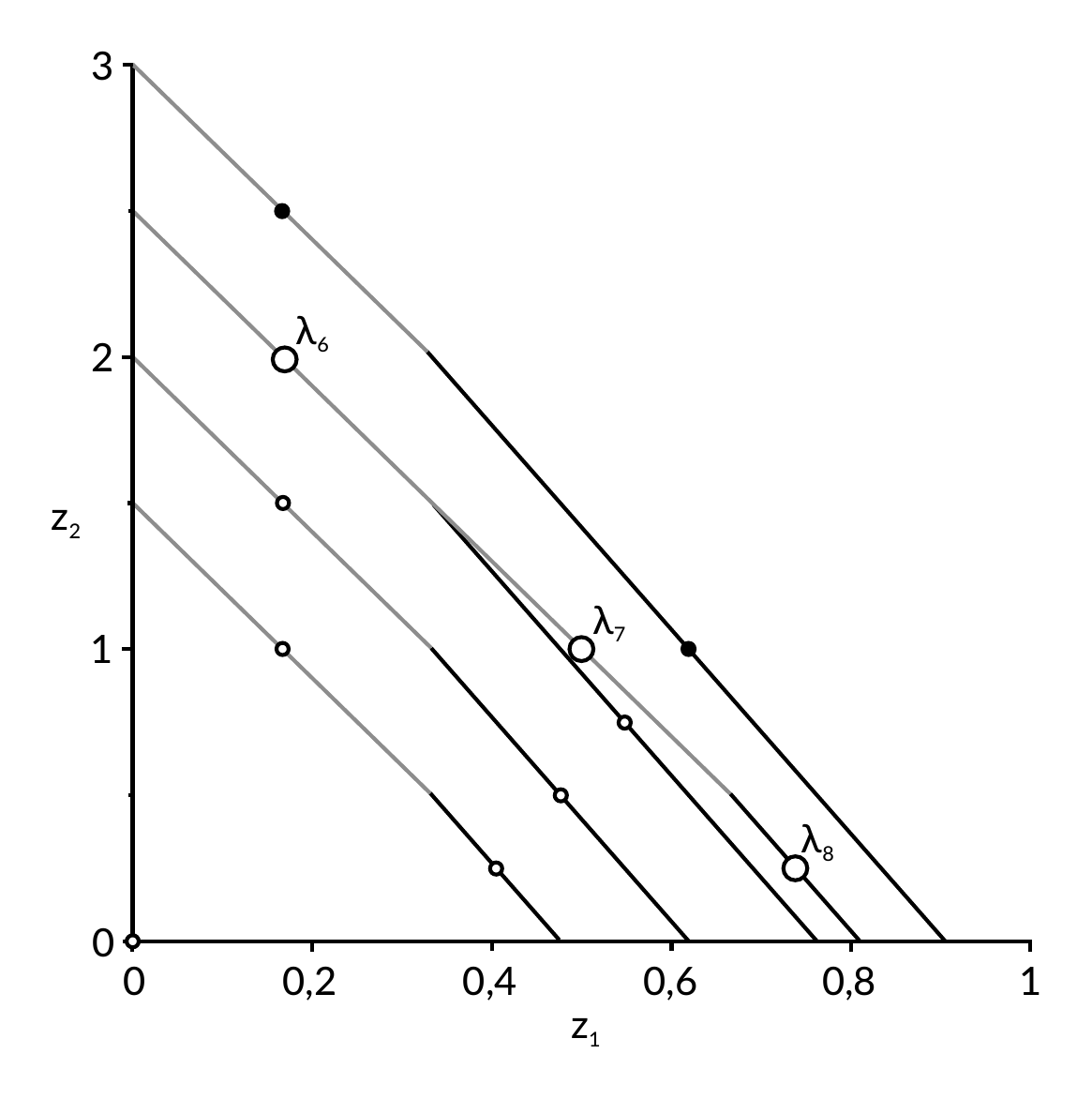}
 \end{center}
\caption{Constancy region associated to $\lambdab_5$
and $\lambdab_6$ (equivalently $\lambdab_7$
and $\lambdab_8$).}
\end{figure}

\vskip 2mm

$\bullet$ {\bf Steps 7 and 8}. The points
$\lambdab_7:=\left(\frac{1}{2},1\right)$ and
$\lambdab_8:=\left(\frac{31}{42},\frac{1}{4}\right)$ satisfy the equality 
$\J(\fab^{\lambdab_8})=\J(\fab^{\lambdab_7})=\J(\fab^{\lambdab_6})$
so they have the same region. We update the sets $N$ and $D$ to
obtain:

\vskip 2mm
\begin{itemize}
\item[$\cdot$]
$N=\{
\left(\frac{1}{6},\frac{5}{2}\right)\left(\frac{13}{21},1\right)
\}$.
\item[$\cdot$] $D= \{\left(0,0\right), \left(\frac{1}{6},1\right), \left(\frac{17}{42},\frac{1}{4}\right),
\left(\frac{1}{6},\frac{3}{2}\right),
\left(\frac{10}{21},\frac{1}{2}\right),\left(\frac{23}{42},\frac{3}{4}\right),
\left(\frac{1}{6},2\right),
\left(\frac{1}{2},1\right),\left(\frac{31}{42},\frac{1}{4}\right)\}$.
\end{itemize}

\end{example}

\section{Jumping divisors}\label{jumping_divisor}

The theory of jumping divisors was introduced in \cite[\S4]{ACAMDC13} in
order to describe the jump between two consecutive multiplier
ideals. The aim of this section is to extend these notions to the
case of mixed multiplier ideals.
More importantly, the theory of jumping divisors is the right framework
that provides the technical results needed in the
proofs of the key results  Theorem \ref{thm:region} and  Proposition \ref{inside_point}.

\vskip 2mm

The proofs of the results that we present in this section are a straightforward extension of
the ones given in \cite[\S4]{ACAMDC13}. However, we include them for completeness.
We  begin with a generalization of the notion of contribution introduced
by Smith and Thompson in \cite{ST07} and further developed by Tucker
in \cite{Tuc10}.

\vskip 2mm

\begin{definition}
Let $\fab:=\left(\fa_1,\ldots,\fa_r\right)\subseteq \left(\cO_{X,O}\right)^r$ be a tuple of ideals, ${\lambdab}
\in \R_{\geqslant0}^r$ a point and $G \leqslant \sum_{i=1}^rF_i$ a  reduced divisor satisfying $\lambda_1 e_{1,i}+\cdots+\lambda_r e_{r,i} - k_i \in \bZ$.
Then it is said that $G$ {\it contributes} to $\lambdab$ if
\[ \pi_{*}\Oc_{X'}(\lceil K_{\pi} -\lambda_1 F_1-\cdots-\lambda_rF_r\rceil+G)\varsupsetneq \J\left({\fab}^{\lambdab}\right)\,.\]
Moreover, this contribution is {\it critical} if for any divisor $0 \leqslant G'< G  $ we have
$$ \pi_{*}\Oc_{X'}(\lceil K_{\pi} -\lambda_1 F_1-\cdots-\lambda_rF_r\rceil+G') = \J({\fab}^{\lambdab}).$$
\end{definition}

\vskip 2mm

The following is the natural extension of \cite[Definition
4.1]{ACAMDC13} to the context of mixed multiplier ideals.

\vskip 2mm

\begin{definition}
 Let ${\lambdab}:=(\lambda_1,\dots,\lambda_r) \in \R_{\geqslant0}^r$ be a jumping point of a tuple of
 ideals $\fab:=\left(\fa_1,\ldots,\fa_r\right)\subseteq \left(\cO_{X,O}\right)^r$. A reduced divisor
 $G \leqslant \sum_{i=1}^r F_i$ for which any $E_j\leqslant G$
 satisfies
\[\lambda_1 e_{1,j}+ \cdots +  \lambda_r e_{r,j} - k_j \in \bZ_{>0}\]
 is called a {\it jumping divisor} for $\lambdab$ if
 \[\J({\fab}^{\lambdapb})= \pi_{*}\Oc_{X'}(\lceil K_\pi- \lambda_1 F_1 - \cdots -\lambda_r F_r \rceil +G)\,,\]
 for any $\lambdapb \in  \{\lambdab - \R_{\geqslant0}^r\} \cap
B_{\varepsilon}(\lambdab)$ for $\varepsilon$ small enough. We say that a jumping divisor is minimal if no proper subdivisor is a
jumping divisor for $\lambdab$, i.e.,
\[\J({\fab}^{\lambdapb})\varsupsetneq \pi_{*}\Oc_{X'}(\lceil K_\pi- \lambda_1 F_1 - \cdots -\lambda_r F_r\rceil+G')\]
for any $ 0\leqslant G'< G $ and for any $\lambdapb \in  \{\lambdab - \R_{\geqslant0}^r\} \cap B_{\varepsilon}(\lambdab)$ for $\varepsilon>0$ sufficiently small.

\end{definition}

\vskip 2mm

Among all jumping divisors we will single out the {\it minimal jumping divisor}
that is constructed following closely Algorithm \ref{A2}.

\begin{definition}
Let $\fab:=\left(\fa_1,\ldots,\fa_r\right)\subseteq
\left(\cO_{X,O}\right)^r$ be a tuple of ideals.
 Given a jumping point $\lambdab \in \R^r_{\geqslant0}$, the corresponding {\it minimal jumping divisor} is the reduced divisor ${G_{\lambda} \leqslant
\sum_{i=1}^r F_i}$ supported on those components $E_j$ for which the
point $\lambdab$ satisfies
$$ \lambda_1  e_{1,j}+ \cdots +   \lambda_r e_{r,j} = k_j + 1 +
e_j^{(1-\varepsilon){\lambdab}}\,,$$ where, for a sufficiently small $\varepsilon>0$,
$D_{(1-\varepsilon)\lambdab} = \sum e_j^{(1-\varepsilon){\lambdab}}
\hskip 1mm E_j$ is the antinef closure of $$\lfloor
(1-\varepsilon){\lambda_1} F_1 + \cdots + (1-\varepsilon){\lambda_r}
F_r - K_\pi\rfloor.$$

\end{definition}

\begin{remark} \label{minimal_hyperplanes}
A jumping point $\lambdab$ is contained in some $\mathcal{C}$-facets of $\cC_{\fab}((1-\varepsilon)\lambdab)$.
The exceptional components $E_j$  such that $H_j: \hskip 1mm \lambda_1  e_{1,j}+ \cdots +   \lambda_r e_{r,j} = k_j + 1 +
e_j^{(1-\varepsilon){\lambdab}}$ are the supporting hyperplanes of these $\mathcal{C}$-facets  are
precisely the components of the minimal jumping divisor
$G_{\lambdab}$.
\end{remark}

\begin{remark} \label{min_leq_max}
For $\varepsilon >0$ small enough we have
$$ G_{\lambdab} \leqslant \lceil K_\pi-(1-\varepsilon)\lambda_1 F_1-\cdots -(1-\varepsilon)\lambda_r F_r\rceil -
\lceil K_\pi-\lambda_1 F_1-\cdots-\lambda_r F_r \rfloor.$$

\end{remark}

The minimal jumping divisor  is not only related to a jumping point, indeed
we can associate it to the interior of each $\cC$-facet.

\begin{lemma}\label{lem:Constancy_facet}
The interior points of a $\cC$-facet have the same minimal jumping divisor.
\end{lemma}

\begin{proof}
This is a direct consequence of Remark \ref{minimal_hyperplanes}.
\end{proof}

\vskip 2mm

 We will prove next that $G_{\lambdab}$ is a jumping divisor
and deserves its name:

\begin{proposition}
 Let $\lambdab$ be a jumping point of a tuple of ideals $\fab:=\left(\fa_1,\ldots,\fa_r\right)\subseteq \left(\cO_{X,O}\right)^r$.
 Then the reduced divisor $G_{\lambdab}$ is a jumping divisor.
\end{proposition}


\begin{proof}
Using Remark \ref{min_leq_max} we have
$$\lceil K_\pi-(1-\varepsilon)\lambda_1 F_1-\cdots -(1-\varepsilon)\lambda_r F_r\rceil  \geqslant
\lceil K_\pi-\lambda_1 F_1-\cdots-\lambda_r F_r \rfloor+G_{\lambdab} $$ for a sufficiently small $\varepsilon>0$ and therefore
$$\J({\fab}^{(1 -\varepsilon)\lambdab})\supseteq \pi_{*}\Oc_{X'}(\lceil K_\pi- \lambda_1 F_1-\cdots-\lambda_r F_r\rceil +G_{\lambdab}).$$

\vskip 2mm

For the reverse inclusion, let ${D}_{(1-\varepsilon)\lambdab}=\sum
e_i^{(1-\varepsilon)\lambdab} E_i$ be the antinef closure of
$$\lfloor(1-\varepsilon)\lambda_1 F_1+\cdots+(1-\varepsilon)\lambda_r F_r - K_\pi\rfloor.$$ We want to check
that $\lfloor\lambda_1 F_1 + \cdots + \lambda_r F_r - K_\pi\rfloor-G_{\lambdab} \leqslant
{D}_{(1-\varepsilon)\lambdab}$. For this purpose we consider two
 cases.
\begin{itemize}
 \item[$\cdot$] If $E_i \leqslant G_{\lambdab}$  then we have  $-k_i + \lambda_1 e_{1,i}+\cdots + \lambda_r e_{r,i} = 1+e_i^{(1-\varepsilon)\lambdab}$.
 And, in particular  \[ \lfloor \lambda_1 e_{1,i}+\cdots +\lambda_r e_{r,i} -k_i\rfloor -1 =  e_i^{(1-\varepsilon)\lambdab}\,.\]
 \item[$\cdot$] If  $E_i \not\leqslant G_{\lambdab}$ then we have $-k_i + \lambda_1 e_{1,i}+\cdots + \lambda_r e_{r,i} < 1+e_i^{(1-\varepsilon)\lambdab}$.
 Thus \[{ \lfloor \lambda_1 e_{1,i}+\cdots +\lambda_r e_{r,i} -k_i\rfloor < 1 + e_i^{(1-\varepsilon)\lambdab}}\] and the
 result follows.
\end{itemize}
 \end{proof}

 \begin{theorem} \label{thm:jump1MMI}
Let $  \lambdab$ be a jumping point of a tuple of ideals $\fab:=\left(\fa_1,\ldots,\fa_r\right)\subseteq \left(\cO_{X,O}\right)^r$. Any reduced contributing divisor $G\leqslant \sum_{i=1}^rF_i$ associated to
$\lambdab$ satisfies either:

\begin{itemize}
 \item[$\cdot$] $\J\left({\fab}^{(1 -\varepsilon)\lambdab}\right)= \pi_{*}\Oc_{X'}(\lceil K_\pi- \lambda_1 F_1-\cdots - \lambda_r F_r\rceil +G)\varsupsetneq \J({\fab}^{\lambdab})$
 if and only if $G_{{\lambdab}}\leqslant G$, or

  \item[$\cdot$] $\J\left({\fab}^{(1 -\varepsilon)\lambdab}\right)\varsupsetneq \pi_{*}\Oc_{X'}(\lceil K_\pi- \lambda_1 F_1-\cdots - \lambda_r F_r\rceil+G)\varsupsetneq \J({\fab}^{\lambdab})$ otherwise.
\end{itemize}
\end{theorem}

\begin{proof}
Since $G \leqslant H_{\lambdab}$, we have $$\lfloor (1 -\varepsilon)\lambda_1
F_1+\cdots+ (1 -\varepsilon)\lambda_r F_r - K_\pi\rfloor  \leqslant \lfloor\lambda_1 F_1+\cdots+\lambda_r F_r - K_\pi\rfloor-G $$ and
therefore
$$\J({\A}^{(1 -\varepsilon)\lambdab})\supseteq \pi_{*}\Oc_{X'}(\lceil K_\pi- \lambda_1 F_1-\cdots-\lambda_rF_r\rceil +G).$$

\vskip 2mm

Now assume $G_{\lambdab}\leqslant G$. Then $\lfloor\lambda_1 F_1+\cdots+\lambda_rF_r -
K_\pi\rfloor-G \leqslant \lfloor\lambda_1F_1+\cdots+\lambda_rF_r - K_\pi\rfloor-G_{\lambdab}$,
and using the fact that $G_{\lambdab}$ is a jumping divisor we obtain
the equality $$\J\left({\fab}^{(1-\varepsilon)\lambdab}\right)=
\pi_{*}\Oc_{X'}(\lceil K_\pi- \lambda_1 F_1+\cdots+\lambda_rF_r\rceil +G).$$

\vskip 2mm

If $G_{\lambdab} \not\leqslant G$, we may consider a component
$E_i\leqslant G_{\lambdab}$ such that $E_i\not\leqslant G$. Notice
that we have
\begin{align*}
 v_i({D}_{(1-\varepsilon)\lambdab}) & =e_i^{(1-\varepsilon)\lambdab}= \lambda_1 e_{1,i}+\cdots+\lambda_r e_{r,i}-k_i-1 \\
 &< \lambda_1 e_{1,i}+\cdots+\lambda_r e_{r,i}-k_i = v_i(\lfloor\lambda_1 F_1+\cdots+\lambda_rF_r - K_\pi\rfloor-G)\,,
\end{align*}

\noindent where ${D}_{(1-\varepsilon)\lambdab}=\sum e_i^{(1-\varepsilon)\lambda} E_i$
is the antinef closure of $\lfloor (1-\varepsilon)\lambda_1 F_1+\cdots+ (1-\varepsilon)\lambda_r F_r -
K_\pi\rfloor$. Therefore, by Corollary \ref{semicont}, we get the
strict inclusion
$$\J({\A}^{(1-\varepsilon)\lambdab })\varsupsetneq \pi_{*}\Oc_{X'}(\lceil K_\pi- \lambda_1F_1-\cdots-\lambda_rF_r\rceil+G).$$
\end{proof}

From this result we deduce the unicity of the minimal jumping divisor.

\begin{corollary}\label{cor:unic_min_jdMMI}
Let $  \lambdab$ be a jumping point of a tuple of ideals $\fab:=\left(\fa_1,\ldots,\fa_r\right)\subseteq \left(\cO_{X,O}\right)^r$. Then  $G_{\lambdab}$ is the unique minimal jumping divisor associated to $\lambdab$.
\end{corollary}

\vskip 2mm

The minimal jumping divisor also allows to describe the jump of mixed multiplier ideals in the other
direction, although in this case we do not have minimality for the jump.

\vskip 2mm

\begin{proposition} \label{prop:jump2_MMI}
Let $\lambdab$ be a jumping point of a tuple of ideals $\fab \subseteq (\cO_{X,O})^r$ and $D_{(1-\varepsilon)\lambdab}$ be the
antinef closure of $\lfloor(1-\varepsilon)\lambda_1 F_1+\cdots+(1-\varepsilon)\lambda_rF_r - K_\pi \rfloor$. Then we
have:

\begin{itemize}
 \item[i)]  $\J(\fab^{(1-\varepsilon)\lambdab})\varsupsetneq \pi_{*}\Oc_{X'}(-D_{(1-\varepsilon)\lambdab}-G_{\lambda})= \J(\fab^{\lambdab})$.

 \item[ii)] $\J(\fab^{(1-\varepsilon)\lambdab})\varsupsetneq \pi_{*}\Oc_{X'}(\lceil K_\pi -(1-\varepsilon)\lambda_1 F_1-\cdots-(1-\varepsilon)\lambda_rF_r\rceil - G_{\lambda})= \J(\fab^{\lambdab})$
\end{itemize}

\end{proposition}

\begin{proof} Let
 $D_{\lambdab}= \sum e_i^{\lambdab}E_i$ be the antinef closure of
 $\lfloor \lambda_1 F_1+\cdots+\lambda_rF_r - K_\pi \rfloor$.

\vskip 2mm

i) Since $G_{\lambdab}$ is a jumping divisor we have $\lfloor \lambda_1 F_1+\cdots+\lambda_rF_r
- K_\pi \rfloor -G_\lambda \leqslant D_{(1-\varepsilon)\lambdab}$, and
 hence $\lfloor \lambda_1 F_1+\cdots+\lambda_rF_r - K_\pi \rfloor  \leqslant D_{(1-\varepsilon)\lambdab} + G_{\lambdab}$. This gives the inclusion
 $$\pi_{*}\Oc_{X'}(-D_{(1-\varepsilon)\lambdab}-G_{\lambdab})\subseteq \J(\fab^{\lambdab}).$$

 \vskip 2mm

In order to check the reverse inclusion
$\pi_{*}\Oc_{X'}(-D_{(1-\varepsilon)\lambdab}-G_{\lambdab})\supseteq
\J(\fab^{\lambdab})$, it is enough, using Corollary \ref{semicont},
to prove $v_i(D_{(1-\varepsilon)\lambdab}+G_{\lambdab}) \leqslant
v_i(D_{\lambdab})=e_i^{\lambdab}$ for any component $E_i$. We have
$e_i^{(1-\varepsilon)\lambdab} \leqslant e_i^{\lambdab}$ just because
$\J(\fab^{(1-\varepsilon)\lambdab})\varsupsetneq  \J(\fab^{\lambdab})$ and the
inequality is strict when $E_i \leqslant G_{\lambdab}$, so the result
follows.

\vskip 2mm

ii) Let $D'$ be the antinef closure of $\lfloor (1-\varepsilon)\lambda_1 F_1+\cdots+(1-\varepsilon)\lambda_rF_r - K_\pi \rfloor + G_{\lambdab}$. Since $G_{\lambdab}
\leqslant H_{\lambdab}$ we have

 $$ \lfloor (1-\varepsilon)\lambda_1 F_1+\cdots+(1-\varepsilon)\lambda_rF_r - K_\pi \rfloor + G_{\lambdab}
 \leqslant \lfloor (1-\varepsilon)\lambda_1 F_1+\cdots+(1-\varepsilon)\lambda_rF_r - K_\pi \rfloor \leqslant D_{\lambdab}$$

  \vskip 2mm

\noindent so the inclusion $\pi_{*}\Oc_{X'}(\lceil K_\pi
-(1-\varepsilon)\lambda_1 F_1-\cdots-(1-\varepsilon)\lambda_rF_r \rceil - G_{\lambdab})\supseteq
\J(\fab^{\lambdab})$ holds. In order to prove the reverse inclusion, we
will introduce an auxiliary divisor $D=\sum d_iE_i \in \Lambda$
defined as follows:

  \vskip 2mm

\begin{tabular}{ll}
 $\cdot$  $d_i= \lfloor (1-\varepsilon)\lambda_1 e_{1,i}+\cdots+(1-\varepsilon)\lambda_re_{r,i}-k_i\rfloor +1$&  if $E_i\leqslant G_{\lambdab}$,\\
 $\cdot$  $d_i= e_i^{(1-\varepsilon)\lambdab}$ & if  $E_i\leqslant H_{\lambdab}$ but $E_i\not\leqslant G_{\lambdab}$,\\
 $\cdot$  $d_i= \lfloor (1-\varepsilon)\lambda_1 e_{1,i}+\cdots+(1-\varepsilon)\lambda_re_{r,i}-k_i\rfloor $ & otherwise.
\end{tabular}

  \vskip 2mm

  Clearly we have  $\lfloor (1-\varepsilon)\lambda_1 F_1+\cdots+(1-\varepsilon)\lambda_rF_r - K_\pi \rfloor + G_{\lambdab} \leqslant D$, but we also have
$\lfloor (1-\varepsilon)\lambda_1 F_1+\cdots+(1-\varepsilon)\lambda_rF_r - K_\pi \rfloor  \leqslant D$. Indeed,

  \vskip 2mm

\begin{itemize}
   \item[$\cdot$]  For $E_i\leqslant G_{\lambdab}$ we have
   \begin{align*}
    \lfloor \lambda_1 e_{1,i}+\cdots+\lambda_re_{r,i} -k_i \rfloor &= \lambda_1 e_{1,i}+\cdots+\lambda_re_{r,i} -k_i \\
    &= \lfloor (1-\varepsilon)\lambda_1 e_{1,i}+\cdots+(1-\varepsilon)\lambda_re_{r,i}-k_i\rfloor +1= d_i\,.
   \end{align*}

   \item[$\cdot$]  If $\lambda$ is a candidate for $E_i$ but $E_i\not\leqslant G_{\lambdab}$,
   \begin{align*}
     \lfloor \lambda_1 e_{1,i}+\cdots+\lambda_re_{r,i} -k_i  = \lambda_1 e_{1,i}+\cdots+\lambda_re_{r,i} -k_i  < 1 + e_i^{(1-\varepsilon)\lambdab}\,,
   \end{align*}

  \noindent  hence $\lfloor \lambda_1 e_{1,i}+\cdots+\lambda_re_{r,i} -k_i \rfloor \leqslant e_i^{(1-\varepsilon)\lambdab}=d_i\,.$

   \item[$\cdot$]  Otherwise
   \begin{align*}
    \lfloor \lambda_1 e_{1,i}+\cdots+\lambda_re_{r,i} -k_i\rfloor =
    \lfloor (1-\varepsilon)\lambda_1 e_{1,i}+\cdots+(1-\varepsilon)\lambda_re_{r,i}-k_i\rfloor = d_i\,.
   \end{align*}

  \end{itemize}

    \vskip 2mm

  Therefore, taking antinef closures, we have $D' \leqslant D_{\lambdab} \leqslant \widetilde{D}$. On the other hand $D\leqslant D'$.
  Namely, $v_i(D')\geqslant e_i^{(1-\varepsilon)\lambdab}$ at any $E_i$ because
$$\lfloor (1-\varepsilon)\lambda_1 F_1+\cdots+(1-\varepsilon)\lambda_rF_r - K_\pi \rfloor
\leqslant \lfloor (1-\varepsilon)\lambda_1 F_1+\cdots+(1-\varepsilon)\lambda_rF_r - K_\pi \rfloor + G_{\lambdab}\,.$$

  Moreover, $v_i(D')\geqslant \lfloor (1-\varepsilon)\lambda_1 e_{1,i}+\cdots+(1-\varepsilon)\lambda_re_{r,i}-k_i\rfloor +\delta_i^{G_{\lambdab}}$ by definition of antinef closure. Here,
  $\delta_i^{G_{\lambdab}}=1$ if $E_i\leqslant G_{\lambdab}$ and zero otherwise. Thus $v_i(D')\geqslant v_i(D)$ as desired.
 As a consequence $\widetilde{D} \leqslant D'$, which, together with the previous  $D' \leqslant D_{\lambdab} \leqslant \widetilde{D}$,
 gives $\widetilde{D} = D'=  D_{\lambdab}$ and the result follows.
\end{proof}

\subsection{Geometric properties of minimal jumping divisors in the dual graph}

We proved in \cite[Theorem 4.17]{ACAMDC13} that minimal jumping divisors associated to
satisfy some geometric conditions in the dual graph in the case of multiplier ideals.
The same properties hold for mixed multiplier ideals.
More interestingly, the forthcoming
Theorem \ref{thm:geo_dual_graphMMI} is the key result that we need in the proof of Theorem \ref{thm:region}.

\begin{lemma} \label{lem:numMMI}
Let $\lambdab$ be a jumping point of a tuple of ideals
 $\fab:=\left(\fa_1,\ldots,\fa_r\right)\subseteq \left(\cO_{X,O}\right)^r$.
For any component $E_i\leqslant G_{\lambdab}$ of the
minimal jumping divisor $G_{\lambdab}$ we have:

\begin{multline*}
\left(\left\lceil K_\pi-\lambda_1 F_1-\cdots-\lambda_r F_r\right\rceil + G_{\lambdab}\right)\cdot E_i 
\\= -2 + \lambda_1 \rho_{1,i}+\cdots+ \lambda_r \rho_{r,i} + a_{G_{\lambdab}}\left(E_i\right) +
 \sum_{ E_j \in \adj(E_i)} \left\{\lambda_1 e_{1,j}+\cdots+\lambda_r e_{r,j}-k_j\right\}\,.
 \end{multline*}
where $\adj(E_i)$ denotes the adjacent components of $E_i$ in the dual graph.
\end{lemma}

\begin{proof}

For any $E_i\leqslant G_{\lambdab}$  we have:

\vskip 2mm

$(\left\lceil K_{\pi}-\lambda_1 F_1-\cdots-\lambda_rF_r\right\rceil + G_{\lambdab})\cdot E_i $

\vskip 2mm

$\hskip .5cm =(\left(K_\pi-\lambda_1 F_1-\cdots-\lambda_rF_r\right)+ \left\{-K_\pi+
\lambda_1 F_1+\cdots+\lambda_rF_r\right\} + G_{\lambdab} - E_i + E_i)\cdot E_i $

\vskip 2mm

$\hskip .5cm = \left(K_{\pi}+ E_i\right)\cdot E_i - (\lambda_1 F_1+\cdots+\lambda_rF_r)\cdot E_i +
\left\{\lambda_1 F_1+\cdots+\lambda_rF_r-K_{\pi}\right\}\cdot E_i +
\left(G_{\lambdab}-E_i\right)\cdot E_i.$

\vskip 2mm

Let us now compute each summand separately. Firstly, the adjunction
formula gives $\left(K_{\pi}+E_i\right)\cdot E_i = -2$ because $E_i
\cong \bP^1$. As for the second and fourth terms, the equality
$-(\lambda_1 F_1+\cdots+\lambda_rF_r) \cdot E_i = \lambda_1\rho_{1,i}+\cdots+\lambda_r\rho_{r,i}$ follows from the definition
of the excesses, and clearly $a_{G_{\lambdab}}\left(E_i\right) =
\left(G_{\lambdab}-E_i\right)\cdot E_i$ because $E_i \leqslant
G_{\lambdab}$.
Therefore it only remains to prove that
\begin{equation} \label{eq:decimal-partMMI}
\left\{\lambda_1 F_1+\cdots+\lambda_rF_r-K_{\pi}\right\}\cdot E_i = \sum_{ E_j \in
\adj(E_i)} \left\{\lambda_1 e_{1,j}+\cdots+\lambda_r e_{r,j}-k_j\right\},
\end{equation}
which is also quite immediate. Indeed, writing
\[\left\{\lambda_1 F_1+\cdots+\lambda_rF_r-K_{\pi}\right\} = \sum_{j=1}^\ell\left\{\lambda_1e_{1,j}+\cdots+\lambda_re_{r,j}-k_j\right\}E_j\,,\]
equality (\ref{eq:decimal-partMMI}) follows by observing that (for $j
\neq i$), $E_j \cdot E_i = 1$ if and only if $E_j \in
\adj\left(E_i\right)$, and the term corresponding to $j=i$ vanishes
because we have ${\lambda_1 e_{1,i}+\cdots+\lambda_r e_{r,i}-k_i \in \bZ}$.
\end{proof}


\begin{corollary} \label{cor:integerMMI}
Let $\lambdab$ be a jumping point of a tuple of ideals $\fab \subseteq
\left(\cO_{X,O}\right)^r$. For any component $E_i\leqslant G_{\lambdab}$ of the
minimal jumping divisor $G_{\lambdab}$ we have:
$$
\lambda_1 \rho_{1,i}+\cdots+ \lambda_r \rho_{r,i} + a_{G_{\lambdab}}\left(E_i\right) + \sum_{ E_j \in \adj(E_i)} \left\{\lambda_1 e_{1,j}+\cdots+\lambda_r e_{r,j}-k_j\right\} \in \Z
$$

\end{corollary}

As in the case of multiplier ideals, minimal jumping divisors satisfy a nice numerical condition.

\begin{proposition} \label{prop:Num_conditionMMI}
Let $\lambdab$ be a jumping point of a tuple of ideals $\fab \subseteq
\left(\cO_{X,O}\right)^r$. For any component $E_i\leqslant G_{\lambdab}$ of the
minimal jumping divisor $G_{\lambdab}$, we have
$$\left(\left\lceil K_{\pi}- \lambda_1 F_1-\cdots-\lambda_r F_r\right\rceil +
G_{\lambdab}\right)\cdot E_i\geqslant 0.$$

\end{proposition}

\begin{proof}

Given a prime divisor $E_i\leqslant G_{\lambdab}$, we consider the short exact
sequence
\begin{align*}
&0 \longrightarrow \Oc_{X'}\left(\left\lceil K_{\pi}- \lambda_1 F_1-\cdots -\lambda_rF_r\right\rceil + G_{\lambdab} - E_i\right) \longrightarrow\\
&\qquad\longrightarrow \Oc_{X'}\left(\left\lceil K_{\pi}- \lambda_1 F_1-\cdots -\lambda_rF_r\right\rceil + G_{\lambdab}\right) \longrightarrow \\
&\qquad\qquad\longrightarrow \Oc_{E_i}\left(\left\lceil K_{\pi}- \lambda_1 F_1-\cdots -\lambda_rF_r\right\rceil + G_{\lambdab}\right) \longrightarrow 0
\end{align*}
Pushing it forward to $X$, we get
\begin{align*}
&0 \longrightarrow \pi_*\Oc_{X'}\left(\left\lceil K_{\pi}- \lambda_1 F_1-\cdots -\lambda_rF_r\right\rceil + G_{\lambdab} - E_i\right) \longrightarrow\\
&\qquad\longrightarrow \pi_*\Oc_{X'}\left(\left\lceil K_{\pi}- \lambda_1 F_1-\cdots -\lambda_rF_r\right\rceil + G_{\lambdab}\right) \longrightarrow \\
&\qquad\qquad\longrightarrow H^0\left(E_i,\Oc_{E_i}\left(\left\lceil K_{\pi}-\lambda_1 F_1-\cdots -\lambda_rF_r\right\rceil + G_{\lambda}\right)\right) \otimes \bC_O,
\end{align*}
where $\bC_O$ denotes the skyscraper sheaf supported at $O$ with
fibre $\bC$. The minimality of $G_{\lambdab}$  (see Corollary \ref{cor:unic_min_jdMMI}) implies that

$$\pi_*\Oc_{X'}\left(\left\lceil K_{\pi}- \lambda_1 F_1-\cdots -\lambda_rF_r\right\rceil + G_{\lambdab} - E_i\right) \neq \pi_*\Oc_{X'}\left(\left\lceil K_{\pi}-\lambda_1 F_1-\cdots -\lambda_rF_r\right\rceil + G_{\lambdab}\right).$$

\vskip 2mm

Thus
$H^0\left(E_i,\Oc_{E_i}\left(\left\lceil K_{\pi}- \lambda_1 F_1-\cdots -\lambda_rF_r\right\rceil + G_{\lambdab}\right)\right) \neq 0$, or equivalently
$$\left(\left\lceil K_{\pi}- \lambda_1 F_1-\cdots -\lambda_rF_r\right\rceil +
G_{\lambdab}\right)\cdot E_i\geqslant 0.$$
\end{proof}


With the above ingredients we can provide the desired geometric
property of the minimal jumping divisors when viewed in the dual graph.

\begin{theorem}\label{thm:geo_dual_graphMMI}
Let $G_{\lambdab}$ be the minimal jumping divisor associated to a
jumping point $\lambdab$ of a tuple of ideals $\fab
\subseteq \left(\cO_{X,O}\right)^r$. Then
the ends of a connected component of $G_{\lambdab}$ must be either
rupture or dicritical divisors.
\end{theorem}

\begin{proof}
Assume that an end $E_i$ of a connected component of $G_{\lambdab}$ is
neither a rupture nor a dicritical divisor. It means that $E_i$ has
no excess, i.e., $\rho_{j,i}=0$ for all $E_j$ of the resolution, and that it has one or two adjacent
divisors, say $E_j$ and $E_l$, in the dual graph but at most one of
them belongs to $G_{\lambdab}$.

\vskip 2mm

For the case that $E_i$ has two adjacent divisors $E_j$ and $E_l$,
the formula given in Lemma \ref{lem:numMMI} reduces to

\begin{align*}
& ( \lceil K_\pi- \lambda_1 F_1-\cdots -\lambda_rF_r\rceil + G_{\lambdab})\cdot E_i \\
&\qquad=-2 + \left\{ \lambda_1 e_{1,j} + \cdots  + \lambda_r e_{r,j}-k_j\right\} + \left\{ \lambda_1 e_{1,l} + \cdots  + \lambda_r e_{r,l} -k_l\right\}\\
&\qquad\qquad + \lambda_1 \rho_{1,i} + \cdots + \lambda_r \rho_{r,i}+ a_{G_{\lambdab}}(E_i)\,.
\end{align*}
Then:

\vskip 2mm

\begin{itemize}
\item[$\cdot$] If $E_i$ has valence one in $G_{\lambdab}$, e.g. $E_l \not \leqslant
G_{\lambdab}$, then $$(\lceil K_\pi- \lambda_1 F_1-\cdots -\lambda_rF_r\rceil + G_{\lambdab})\cdot
E_i= -2 + \left\{ \lambda_1 e_{1,l} + \cdots  + \lambda_r e_{r,l}  -k_l \right\}+ 1 < 0\,. $$

\vskip 2mm

\item[$\cdot$] If $E_i$ is an isolated component of $G_{\lambdab}$, i.e., $E_j, E_l \not \leqslant
G_{\lambdab}$, then
\begin{align*}
&( \lceil K_\pi- \lambda_1 F_1-\cdots -\lambda_rF_r\rceil + G_{\lambdab})\cdot E_i\\
& \qquad = -2 + \left\{ \lambda_1 e_{1,j} + \cdots  + \lambda_r e_{r,j}  -k_j\right\} + \left\{ \lambda_1 e_{1,l} + \cdots  + \lambda_r e_{r,l}
-k_l\right\} <0.
\end{align*}

\end{itemize}

\vskip 2mm

If $E_i$ has just one adjacent divisor $E_j$, i.e. $E_i$ is an end
of the dual graph, the formula reduces to
\begin{align*}
& ( \lceil K_\pi- \lambda_1 F_1-\cdots -\lambda_rF_r\rceil + G_{\lambdab})\cdot E_i\\
&\qquad=-2 + \left\{ \lambda_1 e_{1,j} + \cdots  + \lambda_r e_{r,j}-k_j\right\} + \lambda_1 \rho_{1,i} + \cdots + \lambda_r \rho_{r,i}+ a_{G_{\lambdab}}(E_i)\,.
\end{align*}
Therefore,

\begin{itemize}
\item[$\cdot$] If $E_i$ has valence one in $G_{\lambdab}$, then \[( \lceil K_\pi- \lambda_1 F_1-\cdots -\lambda_rF_r\rceil + G_{\lambdab})\cdot E_i= -2 + 1 < 0 \,.\]

\vskip 2mm

\item[$\cdot$] If $E_i$ is an isolated component of $G_{\lambdab}$, then $$( \lceil K_\pi- \lambda_1 F_1-\cdots -\lambda_rF_r\rceil + G_{\lambdab})\cdot E_i= -2 + \left\{
\lambda_1 e_{1,j} + \cdots  + \lambda_r e_{r,j}  -k_j\right\}  <0 .$$
\end{itemize}

\vskip 2mm

In any case we get a contradiction with Proposition
\ref{prop:Num_conditionMMI}.
\end{proof}

As a consequence of this result we can also provide the following refinement of Proposition \ref{prop:Num_conditionMMI}.

\begin{proposition} \label{prop:Num_condition2MMI}
Let $\lambdab$ be a jumping point of a tuple of ideals $\fab
\subseteq \left(\cO_{X,O}\right)^r$. If $E_i\leqslant G_{\lambdab}$ is neither a
rupture nor a dicritical component of the minimal jumping divisor
$G_{\lambdab}$ we have $$\left(\left\lceil K_{\pi}- \lambda_1 F_1 -\cdots - \lambda_r
F_r\right\rceil + G_{\lambdab}\right)\cdot E_i= 0.$$
\end{proposition}

\begin{proof}
Assume that $E_i\leqslant G_{\lambdab}$ is neither a rupture or a
dicritical component. In particular, it is not the end of a
connected component of $G_{\lambdab}$. Thus, $E_i$ has exactly two
adjacent components $E_j$ and $E_l$ in $G_{\lambdab}$, and its excesses
are $\rho_{j,i}=0$ for all $1\leqslant j\leqslant r$. The formula given in Lemma \ref{lem:numMMI} for $G_{\lambdab}$ reduces to
\begin{align*}
 &\left(\left\lceil K_{\pi}- \lambda_1 F_1-\cdots -\lambda_rF_r\right\rceil + G_{\lambdab}\right)\cdot E_i\\
 &\qquad  = -2 + \lambda_1 \rho_{1,i}+\cdots +\lambda_r\rho_{r,i} + \left\{\lambda_1 e_{1,j} + \cdots  + \lambda_r e_{r,j}-k_j\right\} \\
 &\qquad\qquad{}+\left\{\lambda_1 e_{1,l} + \cdots  + \lambda_r e_{r,l}-k_l\right\} + a_{G_{\lambda}}\left(E_i\right)\,.
\end{align*}
Notice that $a_{G_{\lambdab}}(E_i)=2$, and also that \[\left\{\lambda_1 e_{1,j} + \cdots  + \lambda_r e_{r,j}-k_j\right\} = \left\{\lambda_1 e_{1,l} + \cdots  + \lambda_r e_{r,l}-k_l\right\} = 0\,,\] because $E_j$
and $E_l$ are components of $G_{\lambda}$, so finally
\[\left(\left\lceil K_{\pi}- \lambda_1 F_1-\cdots -\lambda_rF_r\right\rceil +
G_{\lambdab}\right)\cdot E_i = 0\,.\]
\end{proof}


\end{document}